\documentclass[AIF,Unicode,manuscript]{cedram}

\usepackage[all]{xypic}
\usepackage{tikz}
\usepackage{tikz-cd}

\def\Lip{{\rm Lip}}
\def\dist{{\rm dist}}

\def\exp{{\rm exp}}
\def\rank{{\rm rank\;}}

\def\setdiff{\setminus}

\def\Hcal{\mathcal{H}}
\def\Acal{\mathcal{A}}
\def\Ccal{\mathcal{C}}

\def\<{\langle}
\def\>{\rangle}

\title[ Open  manifold admitting no PSC metrics]{On open  manifolds admitting no complete metric with positive scalar curvature}
\alttitle{Sur les vari\'et\'es ouvertes n'admettant aucune m\'etrique compl\`ete avec une courbure scalaire positive}

\author{\firstname{Yuguang} \lastname{Shi}}
\address{Key Laboratory of Pure and Applied Mathematics, School of Mathematical Sciences, Peking University, Beijing, 100871, P.\ R.\ China}
\email[Y. Shi]{ygshi@math.pku.edu.cn}

\thanks{Y. Shi is partially supported by National Key R$\&$D Program of China 2020YFA0712800}

\author{\firstname{Jian}\lastname{Wang}}
\address{Department of Mathematics, Stony Brook University, 100 Nicolls Road, Stony Brook, NY 11794, USA; Academy of Mathematics and Systems Science Chinese Academy of Sciences 55 Zhongguancun East Road, Beijing 100190 China}
\email[J. Wang]{jian.wang.4@stonybrook.edu}

\author{\firstname{Runzhang} \lastname{Wu}}
\addressSameAs{1}{Key Laboratory of Pure and Applied Mathematics, School of Mathematical Sciences, Peking University, Beijing, 100871, P.\ R.\ China}
\email[R. Wu]{wrz0415@stu.pku.edu.cn}

\author{\firstname{Jintian}\lastname{Zhu}}
\address {Institute for Theoretical Sciences, Westlake University, 600 Dunyu Road, 310030, Hangzhou, Zhejiang, People's Republic of China}
\email[J. Zhu]{zhujintian@westlake.edu.cn}

\thanks{J. Zhu is partially supported by National Key R\&D Program of China 2023YFA1009900 as well as the startup fund from Westlake University.} 

\altkeywords{Courbure scalaire positive, vari\'et\'es ouvertes, obstruction topologique.}


\begin{abstract} 
    In this paper, we investigate the topological obstruction problem for positive scalar curvature and uniformly positive scalar curvature on open manifolds. We present a definition for open Schoen-Yau-Schick manifolds and prove that there is no complete metric with positive scalar curvature on these manifolds. Similarly, we define weak Schoen-Yau-Shick manifolds by analogy, which are expected to admit no complete metrics with uniformly positive scalar curvature.

\end{abstract}

\begin{altabstract}
     Dans cet article, nous \'etudions le probl\`eme d'obstruction topologique li\'e \`a  la courbure scalaire positive et \`a  la courbure scalaire uniform\'ement positive sur les vari\'et\'es ouvertes. Nous proposons une d\'efinition des vari\'et\'es ouvertes de type Schoen-Yau-Schick et d\'emontrons qu'il n'existe aucune m\'etrique compl\`ete poss\'edant une courbure scalaire positive sur ces vari\'et\'es. De mani\`ere analogue, nous d\'efinissons les varietes de type Schoen-Yau-Schick au sens faible, qui n'admettent aucune m\'etrique compl\`ete avec une courbure scalaire uniform\'ement positive.

\end{altabstract}    
    
\begin{document}
\maketitle

\section{Introduction}

In Riemannian geometry, a fundamental question revolves around determining which types of open manifolds do not support complete metrics with positive scalar curvature (\emph{hereafter, PSC}) or uniformly positive scalar curvature (\emph{hereafter, UPSC}). Recent progress has focused on the extension of topological obstructions on closed manifolds to non-compact cases, a summary of which is provided below.

A common method for generating non-compact manifolds from closed ones involves performing connected sum operations. The initial exploration of obstructions to PSC on such manifolds stemmed from the Liouville theorem in conformal geometry. Lesourd, Unger, and Yau \cite{LUY2020} established a link between the Liouville theorem and the generalized Geroch conjecture, which states that $\mathbb{T}^n\#M$ cannot admit complete PSC metrics for any manifold $M$. Initially, this conjecture was confirmed in three dimension using the minimal surface method, later refined by Chodosh and Li \cite{2020Generalized} up to dimension seven through Gromov's $\mu$-bubble method. Notably, Wang and Zhang \cite{wz2022} resolved the generalized Geroch conjecture with an additional spin assumption for all dimensions. Analogous conjectures arise when replacing the $n$-torus $\mathbb{T}^n$ with certain generalizations such as the Schoen-Yau-Schick (SYS) manifolds \cite{2020Generalized, Chen2022AGO} or aspherical manifolds \cite{chen2023}.

In addition to connected sum operations, various other techniques exist for constructing non-compact manifolds from closed ones, including submanifold removal or product formations. For instance, Gromov and Lawson demonstrated in  \cite[Example 6.9]{Gr1983}  that removing a sub-torus from the $n$-torus $\mathbb{T}^n$ results in an obstruction of any complete PSC metric.
Conversely, Cecchini, Räde, and Zeidler \cite{Cecchini2023} proved, for $n\leq 7$ but excluding $n=4$, that $M^n\times \mathbb{R}$ does not have complete PSC metrics when $M$ itself is a closed manifold without any PSC metric. This result offers a positive resolution to the Rosenberg-Stolz conjecture.

Furthermore, it is anticipated that the topological obstructions for PSC remain invariant under domination maps, i.e., maps between two manifolds with non-zero degrees (in this paper, we only consider maps of degrees $\pm 1$). For instance, Gromov proposed the ``non-compact domination conjecture'' \cite[Section 4.7]{Gr2023}, suggesting that if a compact orientable manifold cannot be dominated by compact manifolds with $\rm{Sc} >0$, then it cannot be dominated by complete manifolds with $\rm{Sc} >0$. Examples include closed enlargeable manifolds and closed SYS manifolds, which cannot be dominated by a manifold with a complete PSC metric.

In a broader sense, Gromov also introduced the positive scalar curvature domination problems \cite[Section 1.5]{Gr2023}, posing questions about which spaces $M^n$ and classes $h \in H_m(M)$ \footnote{If not specified otherwise, all homology groups are with integral coefficients.} can or cannot be dominated by complete Riemannian manifolds with $\rm{Sc} >0$. If $M^n$ is an aspherical $5$-manifold, He and Zhu \cite{He2023} proved that any class $h$ in $H_4(M,\mathbb Q)$ can not be dominated by a compact manifold with $\rm{Sc} >0$. Conversely, submanifolds can also prevent the existence of PSC metrics on ambient manifolds. As shown in \cite{schoen1979existence,schoen1982complete, Gr1983, chen2021incompressible}, some incompressible hypersurfaces can be PSC obstructions.

To summarize, the aforementioned research efforts aim to understand how PSC obstructions can persist from closed manifolds to non-compact manifolds. In this paper, we aim to expand upon this compact-to-noncompact paradigm in SYS manifolds.

\subsection{Generalization of SYS} Closed SYS manifold was initially introduced by Schoen and Yau  \cite{SY1979} for the dimension-reduction argument and later applied by  Schick \cite{schick1998counterexample} to construct a counterexample to the unstable Gromov-Lawson-Rosenberg conjecture. 

Extending this line of inquiry to open manifolds aligns with the compact-to-noncompact philosophy. Gromov \cite[Section 5.10]{Gr2023} notably outlined a potential extension of Schoen-Yau's dimension-reduction argument to complete open manifolds, hinting at the possibility of defining open SYS manifolds. Nevertheless, the prospect faces uncertainties, particularly concerning the existence of area-minimizing hypersurfaces within each Borel-Moore homology class. Thus, we explore an alternative approach rooted in the concept of exhaustion.

Recall that a closed orientable manifold $M^n$ is said to be \emph{SYS} if there are cohomology classes $\beta_1,\ldots,\beta_{n-2} \in H^1(M)$ such that the homology class 
$$
\tau:=[M]\frown(\beta_1\smile \cdots  \smile\beta_{n-2})\in H_2(M)
$$
does not lie in the image of the Hurewicz map $\pi_2(M)\to H_2(M)$. Such a class $\tau$ is called aspherical. 

To introduce the notion of SYS for open manifolds,   a treatment for ends at infinity is essential.
Let $M$ be a manifold equipped with an exhaustion
\begin{equation*}
	K_1 \subseteq K_2 \subseteq \cdots \subseteq K_i \subseteq \cdots 
\end{equation*}
 by compact sets of $M$. Then 
 an {\it end} of $M$ means a decreasing sequence
\begin{equation*}
	U_1 \supseteq U_2 \supseteq \cdots \supseteq U_i \supseteq \cdots 
\end{equation*} 
where $U_i$ is a connected component of $M \setminus K_i$. 

Define 
$
\mathcal B(M)$ as the collection of subsets $V$ of $M$ with compact boundaries. For  any subset $V\in \mathcal B(M)$ we denote $\partial_\infty V$ to be the collection of ends, denoted by $\{U_i\}_{i=1}^\infty$, satisfying $U_i\subset V$ for $i$ sufficiently large. We remark that $V$ can be compact and in this case, $\partial_\infty V$ is empty.

Notice that if $M$ has a decomposition $M=A\cup B$ with $A,B\in \mathcal B(M)$ and $A \cap B$ is compact, then we have the corresponding ends decomposition 
$$\partial_\infty M=\partial_\infty A\sqcup \partial_\infty B.$$
 
For any subset $V\in \mathcal B(M)$ we can  define the chain complex
$$
C_*(M,\partial_\infty V):={\varprojlim}_i C_*(M,V\setminus K_i)
$$
and denote $H_*(M,\partial_\infty V)$ to be the corresponding homology groups. Moreover, 
 the cohomology groups $H^*(M,\partial_{\infty}V)$ can be defined by the same approach, via the direct limit, instead of the inverse limit. 

The non-compact version of the Poincar\'e duality Theorem \ref{pd} states that if $M$ has an ends-decomposition $\partial_\infty M=\partial_\infty A\sqcup \partial_\infty B$ with $A, B\in \mathcal B(M)$, then
the cap product with the fundamental class $[M]$ in $H_n(M,\partial_\infty M)$ gives the isomorphism of Poincaré dual
\begin{equation*}
	D_M : H^k(M,\partial_\infty A) \to H_{n-k}(M,\partial_\infty B), \quad \alpha \mapsto [M] \frown \alpha.
\end{equation*}

 Throughout this paper, $V$ is always denoted to be a set in $\mathcal B(M)$ and use these notions
$$\partial_s M:=\partial_\infty V \mbox{ and }\partial_b M:= \partial_\infty V^c.$$
where $V^c$ is denoted to be the complement of $V$ in $M$. This gives rise to an ends-decomposition $\partial_{\infty} M = \partial_s M \sqcup \partial_b M $ induced by $V$. Notice that if $\overline{V_1 \Delta V_2}$ is compact, then they induce the same ends-decomposition, and this condition is equivalent to $\partial_{\infty} V_1= \partial_{\infty} V_2$.

\vspace{2mm}

With the above preparation, we introduce the notion of aspherical classes.

\begin{definition}\label{spherical}
Given an ends-decomposition $\partial_{\infty}M=\partial_s M \sqcup \partial_b M$ induced by $V\in \mathcal{B}(M)$, a homology class $\tau \in H_2(M,\partial_s M )$ is called {\it spherical} if for any closed subset $\Omega$ of $M$ with $\overline{\Omega \Delta V^c}$ is compact, the restricted class $\tau|_\Omega\in H_2(M,M\setminus \Omega)$ is a class in the image of the Hurewicz map $\pi_2(M,M\setminus \Omega)\to H_{2}(M,M\setminus \Omega)$. Otherwise, we say that $\tau$ is {\it aspherical}.
\end{definition}

\begin{figure}[tbp]
    \includegraphics[width=1.0\linewidth]{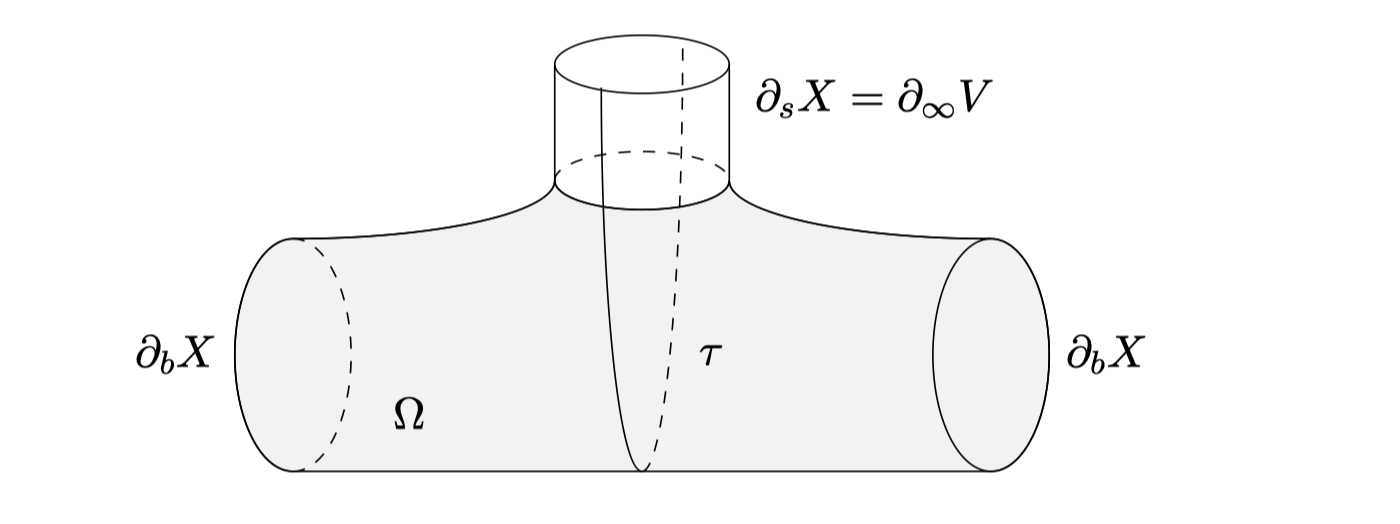}
    \caption{Definition \ref{spherical}}
\end{figure}

It is worth explaining here how the restricted class $\tau |_{\Omega}$ is defined. Intuitively, $\partial_s M = \partial_{\infty} V$ can be regarded as a subset of $M\setminus \Omega$ at infinity, it makes sense that the inclusion $(M, \partial_s M) \hookrightarrow (M, M\setminus \Omega)$ induces the restriction maps
\begin{equation*}
	H_*(M,\partial_s M) \to H_*(M, M \setminus \Omega) .
\end{equation*}

Strictly speaking, there are natural maps from $ \varprojlim H_*(M, V \setminus K_i)$ to $H_*(M, V \setminus K_i)$ by the universal property of inverse limit. Moreover, by the Milnor exact sequence in Lemma \ref{MilnorSeq}, we get natural maps from $H_*(M, \partial_s M)$ to $ \varprojlim H_*(M, M \setminus K_i)$. Finally, we can always choose a $K_i$ such that $V \setminus K_i \subseteq M \setminus \Omega$ to get the natural restrictions $H_*(M, V \setminus K_i) \to H_*(M, M \setminus \Omega)$.

\begin{remark}The class $\tau$ at infinity is supported in $\partial_s M$, instead of $\partial_b M$. Since $\partial_s M$ serves like a side-face in the picture, we use $\partial_s M$ to mean the side-ends of $M$ and $\partial_b M$ to mean the bottom-ends of $M$.  
\end{remark}

When considering finer structures of ends, it is necessary and crucial to use the relative homology $H_2(M,\partial_s M)$ instead of $H_2(M, \partial_{\infty} M)$, because some surfaces are non-trivial in $H_{2}(M, \partial_s M)$, but vanish in $H_2(M, \partial_\infty M)$.

\vspace{2mm}

Now we define the generalized SYS manifolds under our consideration as follows.

\begin{definition}\label{SYSopen}
	An orientable manifold $M^n$ (possibly open) is said to be  \emph{SYS} if there is a $V\in \mathcal{B}(M)$ and a decomposition $\partial_{\infty}M=\partial_s M \sqcup \partial_b M$  such that  there is a cohomology class $\beta_1$ in $H^1(M,\partial_b M)$ and some cohomology classes $\beta_2, \ldots,  \beta_{n-2}$ in $H^1(M)$ such that the class
	\begin{equation*}
		[M] \frown ( \beta_1 \smile\cdots \smile \beta_{n-2}) \in H_2(M,\partial_s  M)
	\end{equation*}
	is aspherical.
\end{definition}

\begin{remark}
If $M$ is closed, this definition coincides with the original one, and it is well-defined without any confusion.
\end{remark}

Our main result is stated as follows.

\begin{theorem}\label{main}
For  $3\leq n\leq 7$, there is no complete metric with positive scalar curvature on an open SYS $n$-manifold.
\end{theorem}

Let us provide several concrete examples for open SYS manifolds, where the verification will be postponed to Section \ref{Sec: examples}. For convenience, we assume that all manifolds are orientable and open unless otherwise stated.

\begin{example}\label{Expl1}
	If $M$ is an open SYS manifold, then $M \# N$ is an open SYS manifold for arbitrary manifold $N$.
	
	 More generally, if $M$ is a domination of an open SYS manifold, then $M$ is open  SYS, where a domination of $\underline{M}$ on $M$ means that there exists a quasi-proper map $f : M \to \underline M$ (see Definition \ref{quaiproper}) with $\deg f =\pm 1$. 
	
\end{example}

\begin{example}\label{Expl2}
	If $M$ is a closed SYS manifold, then $M \times \mathbb R$ is an open SYS manifold.\end{example}

\begin{example}\label{Expl3}
	Let $M^n$ be a closed SYS manifold and $\Gamma$ be a submanifold. If $\Gamma$ satisfies one of the following:
	\begin{itemize}
		\item[(1)] $\dim(\Gamma) \le 1$;
		\item[(2)] The first betti number $b_1(\Gamma) \le n-3$,
	\end{itemize}
	then $M \setminus \Gamma$ is an open SYS manifold.
\end{example}

We use Theorem \ref{main} to partially answer a conjecture of Schoen mentioned in \cite{LM19}. 

\begin{corollary}
Let $M^n$, $3\leq n\leq 7$ be a closed manifold and $\Gamma$ a finite simplicial complex embedded in $M$ of codimension $k \ge \frac{n}{2} +1$. If $M\setminus \Gamma$ is an open SYS manifold and $g$ is a metric in  $ L^\infty(M)\cap C^\infty(M\setminus \Gamma)$ with nonnegative scalar curvature on $M\setminus \Gamma $, then $g$ is Ricci-flat on $M\setminus \Gamma$.
\end{corollary}

The original conjecture of Schoen \cite{LM19} supposes that $\Gamma$ is co-dimensional $3$ and the Yamabe invariant of $M$ is non-positive, instead of the open SYS property of $M$. The lower bound of codimension $k$ ensures that the conformal metric will be complete (see the computation in \cite[Proposition 2.4]{wang2024scalarcurvaturerigidityspheres}). Mantoulidis-Li \cite{LM19} and Kazaras \cite{kazaras2019} resolved the conjecture for 3-dimension and 4-dimension respectively.

\subsection{Uniformly Positive Scalar Curvature}
PSC and UPSC obstructions exhibit differing behaviors in non-compact manifolds. For example,  $\mathbb R^2$ admits a complete PSC metric, but there is no complete UPSC metric. Hence, Definitions \ref{spherical} and \ref{SYSopen} fall short in determining which types of open manifolds admit a complete UPSC metric.

We modify the previous definition slightly and introduce the so-called ``strongly spherical '' class.
\begin{definition}\label{Defwclass}

	Given an ends-decomposition $\partial_{\infty}M=\partial_s M \sqcup \partial_b M$ induced by $V\in \mathcal{B}(M)$, a homology class $\tau \in H_2(M, \partial_s M)$ is called {\it strongly spherical} if for any closed subset $\Omega$ of $M$ satisfying that $\overline{\Omega \Delta V^c}$ is compact , the restricted class $\tau|_\Omega\in H_2(M,M\setminus \Omega)$ is a class in the image of the Hurewicz map $\pi_2(M)\to H_{2}(M,M\setminus \Omega)$. Otherwise, we say that $\tau$ is {\it weakly aspherical}.

\end{definition}

Likewise, we define weak SYS manifolds. 

\begin{definition}\label{DefwSYS}
	An orientable manifold $M^n$ (possibly open) is said to be \emph{weakly SYS} if there is an ends-decomposition $\partial_{\infty}M=\partial_s M \sqcup \partial_b M$ induced by $V$ satisfying the following property: there are a cohomology class $\beta_1$ in $H^1(M,\partial_b M)$ and cohomology classes $\beta_2, \ldots,  \beta_{n-2}$ in $H^1(M)$ such that the class
	\begin{equation*}
		[M] \frown ( \beta_1 \smile\cdots \smile \beta_{n-2}) \in H_2(M,\partial_s  M)
	\end{equation*}
	is weakly aspherical.
\end{definition}

\begin{example}
	$\mathbb T^2 \times \mathbb R^2$ is weakly SYS but not SYS as it admits a complete PSC metric. 
\end{example}

Similarly to Theorem \ref{main}, we can prove the following: 

\begin{theorem}\label{ThmwSYS}
	For $3 \le n \le 7$, there are no complete metrics with uniformly positive scalar curvature on a weakly SYS $n$-manifold. 
\end{theorem}

We will give a class of weak SYS manifolds obtained by deleting a certain submanifold from closed SYS manifolds.

\begin{corollary}
	Let $M$ be a closed SYS manifold with the aspherical class $\tau= [M] \frown (\beta_1 \smile \ldots \smile  \beta_{n-2}) \in H_2(M)$ and $\Gamma \subset M$ be a embedded submanifold. If $\tau$ is not in the subgroup generated by the image of $H_2(\Gamma)$ and $\pi_2(M)$, then $M$ is weakly SYS.
\end{corollary}

\begin{remark}
	If the Hurewicz map $\pi_2(\Gamma) \to H_2(\Gamma)$ is surjective, then the aspherical class $\tau$ automatically satisfies the condition of the above theorem. 
\end{remark}

\subsection{Idea of the proofs of Theorem \ref{main} and Theorem \ref{ThmwSYS}}

 Similar to the compact case, we apply the Schoen-Yau descent argument to find a suitable surface representation of the aspherical class $\tau$ defined in Section 1.1. It will be seen later that both the proofs of Theorem \ref{main} and Theorem \ref{ThmwSYS} follow the same strategy: determining the topological characterization of this surface under the condition of PSC or UPSC. The proof of Theorem \ref{ThmwSYS} is much simpler than Theorem \ref{main} itself, due to the UPSC assumption. Thus, we will only focus on the proof of Theorem \ref{main} here.

\vspace{2mm}

As the non-compactness of manifolds leads to the lack of the existence of global stable minimal surfaces, we will find a ``minimal'' surface in a sufficiently large closed subset $\Omega$  (as chosen in Definition \ref{spherical}) with possible use of Gromov's $\mu$-bubble. Precisely, we construct a stable weighted slicing with free boundary

\[
(\Sigma_{2},\partial\Sigma_{2},w_{2})\to\ldots \to (\Sigma_{n-1},\partial\Sigma_{n-1},w_{n-1})  \to(\Sigma_n,\partial\Sigma_n,w_n)=(\Omega,\partial \Omega,1),
\]
 where $\Sigma_{k}$ is a $k$-dimensional submanifold in $\Omega$, $w_{k}$ is a smooth positive function  on $\Sigma_{k}$ which will be served as a weight later.
 
 \vspace{2mm}
   
 The main difficulty lies in obtaining a topological description of $(\Sigma_2, \partial\Sigma_2)$ under the positive scalar curvature assumption. Specifically, we can give a localized topological characterization (see Proposition \ref{Prop: quantitive topology}) asserting  that 
 \begin{center}
 $(\Sigma_2, \partial \Sigma_2)$ must be a disk far away from $\partial \Sigma_2$.  \quad \quad \quad \emph{(ltc)}\end{center}
 
The reason is as follows: We first apply the warping product trick consecutively and get that $\Sigma_{2}$ (actually $\Sigma_{2} \times \mathbb T^{n-2}$) can be viewed as a surface with a PSC metric. 

Choose a compact set $K\subset \Sigma_2$  such that $\partial \Sigma_2$ is sufficiently far from $K$. If there is a non-trivial closed curve $\gamma \subset K$ in $H_1(\Sigma_2)$, we lift it to a certain covering space and assume that 
\begin{itemize}
\item the curve $\tilde \gamma$ separates two ends of $\partial \tilde \Sigma_2$;
\item it is sufficiently far from $\partial \tilde \Sigma_2$.
\end{itemize}
A width estimate (see Lemma \ref{TStableWidthEstimate}) suggests that  the distance between such a curve and $\partial \tilde \Sigma_2$ must be bounded by a definite constant, which is a contradiction with the second fact. 

We can conclude that each component $\partial K$ is contractible in $\Sigma_2$. Namely, $\Sigma$ is a disc far away from $\partial \Sigma_2$

\vspace{2mm}

Remark that the bound follows from the width estimate in Lemma \ref{WidthEstimate}, which is a quantitative improvement of the fact that a Riemannian surface with complete PSC metrics cannot have two ends. 

\vspace{2mm}

We use the localized topological characterization (see \emph{(ltc)}) and see  that the class $\tau$ is spherical, as $(\Sigma_2, \partial \Sigma_2)$ represents the class $\tau|_{\Omega} $ as in Definition \ref{spherical}. However, $\tau$ is aspherical from the SYS assumption.

 \vspace{2mm}

 We now explain the construction of the co-dimensional one hypersurfaces.

For the case $\partial_{b} M = \emptyset$, the closed subset $\Omega$ is compact. By the standard theorems of geometric measure theory, the minimal surface exists with singularities of codimension $7$.

For the case $\partial_{b} M \neq \emptyset$, $\Omega$ may be non-compact, so we try to find a $\mu$-bubble $\Sigma_{n-1}$ at the first step. An essential advantage of $\mu$-bubbles over minimal hypersurfaces is that the former are easier to ``trap'' and prevent from fully sliding away to infinity than the minimal hypersurfaces. By lifting to a covering space, we can assume that $\Omega$ is a non-compact manifold with boundary $\partial_{s}$ and a nontrivial ends decomposition $\partial_\infty \Omega=\partial_+\sqcup \partial_-$. The existence of $\mu$-bubble is guaranteed by the non-trivial ends decomposition. 

\begin{figure}[tbp]
    \includegraphics[width=1.0\linewidth]{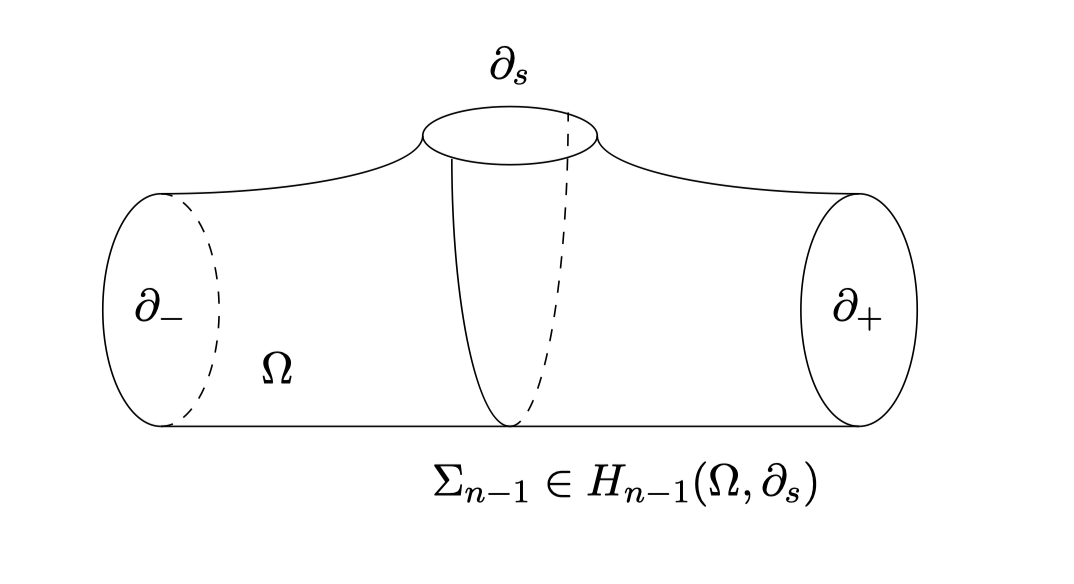}
    \caption{$\mu$-bubble with boundary}
\end{figure}

\vspace{2mm}

\subsection{Arrangement of  this paper} The rest of this paper is organized as follows:
\begin{itemize}
	\item In Section 2, we provide some preliminary content on topology and analysis. 
	\item In section 3, we prove Theorem \ref{main}.
	\item In section 4, we give some examples of open SYS manifolds, including those mentioned in the introduction.
	\item In Section 5, we prove Theorem \ref{ThmwSYS} and discuss the obstruction for UPSC.
\end{itemize}

\section*{Acknowledgement}
We would like to thank Prof. Chao Li for several helpful conversations about non-compact minimal surfaces. We are grateful to Prof. Xin Zhou and Dr. Zhichao Wang for the discussion on minimal surfaces with free boundary. Finally, during the workshop ``Recent Advances in Comparison Geometry (24w5226)'' held in Hangzhou, Prof. Pengzi Miao highlighted a potential link between Theorem \ref{main} and the Schoen conjecture, for which we are grateful. Additionally, we would like to thank the Banff International Research Station and the Institute of Advanced Studies in Mathematics for providing us with this enriching opportunity.

\section{Preliminary}
In this section, we introduce some notions and some useful lemmas. 

\subsection{Topological Preliminaries}

\begin{lemma}\label{pd1}
Let $M^n$ be a connected and oriented smooth manifold, possibly open,  $0\neq \alpha \in H_{n-1}(M^n)$ be a non-zero homology class, then $\alpha$ can be represented by closed embedded oriented hypersurface $\Sigma^{n-1}$ in $M^n$; Moreover, there is a smooth map $f$: $M^n \rightarrow\mathbb{S}^1$ and    $\Sigma^{n-1}$ is regular level set of $f$. 	
\end{lemma} 
\begin{proof}
Let $\eta \in H^1_c(M)$ be the Poincaré dual of $\alpha\in H_1(M)$.  One has that 

$$H_c^1(M):=\varinjlim\limits_{K \subseteq M} H^1(M, M\setminus K),$$
We can find a compact set $K\subset M$ such that $\eta \in H^1(M, M\setminus K) $. Without loss of the generality, we may assume $\partial K$ is smooth.

The excision property of cohomology (see Page \cite{2002Hat}) gives some isomorphisms: $H^1(M,M\setminus K) \cong H^1(M / (M\setminus K))$ and $H^1(\mathbb S^1, \{-1\}) \cong H^1(\mathbb S^1)$. Brown's representability theorem (see Page 448 \cite{2002Hat}) suggests the relative version  as follows:
\begin{center}\begin{tabular}{rccc}
$\tilde{T} :$&$\left[(M , M\setminus K), (\mathbb{S}^1, \{-1\})\right]$ &$\longrightarrow  $&$H^1(M , (M\setminus K))$\\
&$[f]$&$\longrightarrow$ &$f^*(\theta)$\\
\end{tabular}
\end{center}
where $\left[(M ,M\setminus K), (\mathbb{S}^1, \{-1\})\right]$ is the set of homotopic classes of maps from $(M,M\setminus K)$ to $(\mathbb{S}^1, \{-1\})$, and $\theta$ is a generator of $H^1(\mathbb S^1,\{-1\})\cong H^1(\mathbb{S}^1)$.  The map $\tilde{T}$ is bijective.

 For any element $\eta\in H^1(M, M\setminus K)$, there is a smooth function $f: (M, M\setminus K)\rightarrow (\mathbb{S}^1, \{-1\})$ with $f^*(\theta)=\eta$. Sard's Theorem allows us to find a regular value $\alpha\neq -1$ in $ \mathbb{S}^1$ such that $f^{-1}(\alpha)$ is a smooth submanifold in $K$. Moreover, the submanifold $f^{-1}(\alpha)$ is the Poincar\'e duality of $\eta$. 

\end{proof}

\begin{lemma}\label{Covering}
	Let $M^n$ be an oriented smooth Riemannian manifold and $\alpha \in H_{n-1}(M^n)$ be a non-zero class. Then there exists a covering space $\tilde M$ of $M$ associated with a smooth function
	$$
		\rho : \tilde M \to \mathbb R
	$$ 
	satisfying
	\begin{itemize}
		\item[(1)] $ \tilde \Sigma = \rho^{-1}(0)$ is a smooth embedded hypersurface.
		\item[(2)] $p_*[\tilde \Sigma]=\alpha \in H_{n-1}(M)$ where $p : \tilde M \to M$ is the projection map.
		\item[(3)] $\Lip(\rho ) < 1$ and $\rho$ is surjective.
	\end{itemize}
\end{lemma}

\begin{proof}
	Let  $\alpha \in H_{n-1}(M^n)$ be the non-zero class. From  Lemma \ref{pd1},  there is a smooth map $f$: $M^n \rightarrow\mathbb{S}^1$ satisfying that 
\begin{itemize}
\item the regular level set $\Sigma:=f^{-1}(1)$ is  a closed hypersurface; 
\item $[\Sigma]=\alpha$
\end{itemize} We consider the pullback
\begin{center}
\begin{equation}\label{lift1}
\begin{tikzcd}
M \times_{f} \mathbb{R}\arrow{r}{\tilde{f}}\arrow{d}{p_1}	&\mathbb{R} \arrow{d}{p_2}\\
M\arrow{r}{f} &\mathbb{S}^1
\end{tikzcd}
\end{equation}
\end{center}where $M \times_f \mathbb R := \{(x,t) \in M \times \mathbb R  : f(x) = p_2(t)   \}$ is  a covering space of $M$. 

The hypersurface $\tilde \Sigma := \tilde f^{-1}(0)$ is the lifting of $\Sigma$ (i.e. $(p_1)_*([\tilde \Sigma]) = [\Sigma]$), and we take $\tilde M$ to be all connected components of $M \times_f \mathbb R$ that contain $\tilde \Sigma$.

We denote $\tilde M^{\pm} := \tilde f^{-1}(\mathbb R^{\pm}) \cap \tilde M$, both of which are smooth manifolds with boundary $\tilde \Sigma$. Since $[\tilde \Sigma] \neq 0$ in $H_{n-1}(\tilde M)$, both $\tilde M^{\pm}$ are non-compact. Then we define a signed distance function 
\begin{equation*}
	d(x) := 
	\begin{cases}
		\dist(x, \tilde \Sigma) \quad & x\in \tilde M^+ \\
		-\dist(x, \tilde \Sigma) \quad & x \in \tilde M^-
	\end{cases}
\end{equation*}
By modifying the signed distance function $d(x)$, we can construct a proper map $\rho: \tilde M \to \mathbb R$ with $\Lip (\rho) < 1 $. Furthermore, $\rho$ is surjective by construction.

\end{proof}

\subsection{\texorpdfstring{$\mu$-bubbles}{mu-bubble }}

\begin{definition}
	A band is a connected compact manifold $M$ and a decomposition of boundary 
	\begin{equation*}
		\partial M = \partial_- \cup \partial_+
	\end{equation*}
	where $\partial_{\pm}$ are unions of connected components of $\partial M$ and both are non-empty.
\end{definition}

\begin{definition}\label{barrierDefn}
	Given a Riemannian band $(M,g;\partial_-,\partial_+)$, a function $\mu$ is said to satisfy the  \emph{barrier condition}
	if either $\mu\in C^\infty(\mathring M)$ with $\mu\rightarrow \pm\infty$
	on $\partial_{\mp}$, 
	or $\mu\in C^\infty(M)$ with
	\begin{equation}\label{barrierDef}
		\mu|_{\partial_-} > H_{\partial_-},\qquad \mu|_{\partial_+} < H_{\partial_+}
	\end{equation}
	where $H_{\partial_-}$ is the mean curvature of $\partial_-$ with respect to the inward normal and $H_{\partial_+}$ is the mean curvature of $\partial_+$ with respect to the outward normal.
\end{definition}

Choose a Caccioppoli set $\Omega_0$ with smooth boundary $\partial \Omega_0 \subset \mathring{M}$ and $\partial_+ \subset \Omega_0$. Consider the following functional 

\begin{equation}\label{mubbFunc}
	\Acal^\mu_{\Omega_0}(\Omega):= \Hcal^{n-1}(\partial^* \Omega)-\int_M(\chi_{\Omega}-\chi_{\Omega_0}) \mu \,\mathrm d \Hcal^n
\end{equation}
where $\Omega$ is any Caccioppoli set of $M$ with reduced boundary $\partial^* \Omega $ such that $\Omega \Delta \Omega_0 \Subset \mathring{M}$. 

The existence and regularity of a minimizer $\mathcal{A}^{\mu}$ among all Caccioppoli sets was claimed by Gromov in \cite[Section 5.2]{Gr2023}, and was rigorously carried out by the fourth-named author in \cite[Proposition 2.1]{Zhu21}.

\begin{lemma}\label{mubbExReg}{\rm (Cf. \cite[Proposition 2.1]{Zhu21})}
	Let $(M^n,g; \partial_-,\partial_+)$ be a Riemannian band with  $n\le 7$,
	and let $\Omega_0$ be a reference set.
	If $\mu$
	satisfies the barrier condition, 
	then there exists an $\Omega\in \Ccal_{\Omega_0}$ with smooth boundary such that
	$$
		\Acal^\mu_{\Omega_0}(\Omega)
		 = \inf_{\Omega'\in \Ccal_{\Omega_0}} \Acal_{\Omega_0}^\mu(\Omega').
	$$
\end{lemma}

\begin{remark}
	In Lemma~\ref{mubbExReg}
	the smooth hypersurface $\Sigma:=\partial\Omega\setdiff\partial_-$
	is homologous to $\partial_+$. 
\end{remark}

We next discuss the first and second variations of $\mu$-bubble.
\begin{lemma}\label{mubbFirstVar}{\rm (Cf. \cite[lemma 13]{2020Generalized})}
	If $\Omega_t$ is a smooth $1$-parameter family of regions with $\Omega_0 = \Omega$ and normal speed $\psi$ at $t=0$, then 
	\begin{equation*}
		\frac{d}{dt} \Acal^{\mu}(\Omega_t) = \int _{\Sigma_t} (H-\mu ) \psi \, \rm d\Hcal^{n-1}
	\end{equation*}
	where $H$ is the scalar mean curvature of $\partial \Omega_t$. In particular, a $\mu$-bubble $\Omega$ satisfies 
	\begin{equation*}
		H=\mu 
	\end{equation*}
	along $\partial \Omega$.
\end{lemma}

\begin{lemma}\label{mubbSecondVar}{\rm (Cf. \cite[lemma 14]{2020Generalized})}
	Consider a $\mu$-bubble $\Omega$ with $\partial \Omega = \Sigma$. Assume that $\Omega_t$ is a smooth $1$-parameter family of regions with $\Omega_0 = \Omega$ and normal speed $\psi$ at $t=0$, then $\mathcal Q(\psi) := \left. \frac{d^2}{dt^2} \right|_{t=0} (\Acal(\Omega_t)) \ge 0 $ where $\mathcal Q(\psi)$ satisfies 
	\begin{equation*}
		\mathcal Q(\psi) = \int_{\Sigma} |\nabla_{\Sigma} \psi|^2 + \frac{1}{2} {\rm{Sc}}(\Sigma)  \psi^2 -\frac{1}{2}(|\rm{II}_{\Sigma}|^2+{\rm{Sc}}(M) + \mu^2 - 2 \langle \nabla_M \mu , \nu  \rangle ) \psi^2 \, \rm d\Hcal^{n-1}
	\end{equation*} 
	where $\nu$ is the outwards pointing unit normal.
	
\end{lemma}

It is well-known that if the scalar curvature of a torical band $M$ is bounded below by ${\rm{Sc}}(\mathbb{S}^n) = n(n-1)$, then
\begin{equation*}
	{\rm{width}}(M) = \dist(\partial_-,\partial_+) \leq \frac{2\pi	}{n}.
\end{equation*}
However, if the scalar curvature is only bounded below by $\sigma$ on a neighborhood of a separating hypersurface and ${\rm{Sc}}(M) > 0$ (actually, not too negative), then the distance from this hypersurface to the boundary $\partial M$ can not be too large. 

The following lemma also appears in \cite{Cecchini2023, hao2023llarull}.

\begin{lemma}\label{WidthEstimate}
	Let $(M^n,g; \partial_-,\partial_+)$ be a compact Riemannian band with  $n\le 7$ associated with a smooth function
	\begin{equation*}
		\rho : (M,\partial_{\pm}) \to ([-D,D], \pm D)
	\end{equation*}
	satisfying 
	\begin{itemize}
		\item[(1)] $\Lip (\rho ) < 1;$
		\item[(2)] The scalar curvature ${\rm{Sc}}(M)  \ge \sigma >0$ in $\rho^{-1}([-1,1])$ and ${\rm{Sc}}(M) > 0 $ everywhere in $M$.
		\item[(3)] Any closed hypersurface $\Sigma$ in $M$ which separates $\partial_-$ from $\partial_+$ admits no metric with positive scalar curvature.
	\end{itemize}
	Then we have $D \le D_0$ where $D_0=D_0(\sigma) >1$ is a positive constant depending only on $\sigma$.
\end{lemma}

\begin{proof} Suppose the contrary  that $D > 2(1-2h^{-1}_0(1))$, where $h_0(x) = -\sqrt{\sigma} \tan(\frac{\sqrt{\sigma}}{2}x)$.  We first construct a function $\mu$ that satisfies the barrier condition. Then use it to find the $\mu$-bubble $\Sigma$ admitting a metric with positive scalar curvature and separating  $M$, which leads to a contradiction.
	
	Consider the piecewise smooth function
	\begin{equation*}
		h(x):= 
		\begin{cases}
			-\infty  \quad &x \ge 1-2h^{-1}_0(1) \\
			\frac{2}{x-1+2h^{-1}_0(1)} \quad  &1 \le x \le1-2h^{-1}_0(1)\\
			-\sqrt{\sigma } \tan(\frac{\sqrt{\sigma }}{2} x) \quad  &-1 \le x\le 1 \\
			\frac{2}{x+1+2h_0^{-1}(-1)} \quad &  -1-h^{-1}_0(-1) \le x\le-1\\
			+\infty \quad &x \le -1-h^{-1}_0(-1)
		\end{cases}
	\end{equation*}
 Ignore the smoothness of $h$ and choose $\mu: = h \circ \rho $. We directly calculate  that
	\begin{equation*}
		{\rm{Sc}}(M) + \mu ^2 - 2|\nabla_M \mu | > 0.
	\end{equation*}
From the hypothesis,  $\mu$ satisfies the barrier condition. By the Lemma \ref{mubbExReg}, the $\mu$-bubble $\Sigma$ exists and separates $M$. 
	
	We now show that $\Sigma$ admits a positive scalar curvature metric. Combining  with the second variation formula in Lemma \ref{mubbSecondVar}, we have
	\begin{equation*}
		\int_{\Sigma} |\nabla_{\Sigma} \psi |^2 + \frac{1}{2} {\rm{Sc}}(\Sigma) \psi^2 > 0.
	\end{equation*} for any smooth function $\psi$ on $\Sigma$. 
We have the operator 
	\begin{equation*}
		-\Delta_{\Sigma} + \frac{1}{2} {\rm{Sc}}(\Sigma) > 0.
	\end{equation*}
	Hence if $n \ge 3$, the conformal Laplacian $L = -4\frac{n-1}{n-2} \Delta_{\Sigma} + {\rm{Sc}}(\Sigma)$ has positive first eigenvalue. As a conclusion, $\Sigma$ admits a PSC metric. If $n=2$, it is obtained directly by the Gauss-Bonnet theorem.

	Let $\phi : [-D,D] \to [-D,D]$ be a monotonically non decreasing smooth function such that $\phi \equiv \pm 1 $ near the neighborhood of $\pm 1$ and $\Lip (\phi) \le 1+\varepsilon$. The composition $\tilde h = h \circ \phi $ is smooth and we use $\tilde \mu = \tilde h \circ \rho $, instead of $\mu$ and the same argument to complete the proof.
\end{proof}

\begin{definition}\label{wrapped}
Let $(M,g)$ be a Riemannian manifold. Given any function $\sigma:M\to \mathbb R$ we say that $(M,g)$ has its $T^*$-stabilized scalar curvature bounded below by $\sigma$ if there are finitely many smooth positive functions $u_1,\ldots,u_k$ such that the warped metric
$$
g_{warp}=g+\sum_{i=1}^k u_i^2\mathrm d\theta_i^2\mbox{ on }M\times \mathbb T^k
$$
satisfies ${\rm{Sc}}(g_{warp})\geq \sigma\circ \pi_M$, where $\pi_M:M\times \mathbb T^k\to M$ is denoted to be the canonical projection map.
\end{definition}

\begin{remark}\label{no-circle-T-scal}  For example, $\mathbb{S}^1$ has no metric with positive $T^*$-stabilized scalar curvature. The reason is as follows: 

If not, there is a $\mathbb{T}^n$ with positive scalar curvature. However, the work of Gromov-Lawson \cite{Gr1983} suggests that $T^{n+1}$ has no metric with positive scalar curvature, which leads to a contradiction.  

\end{remark}

We generalize the Lemma \ref{WidthEstimate} to $T^*$-stabilized scalar curvature.

\begin{lemma}\label{TStableWidthEstimate}
	Let $(M^n,g; \partial_-,\partial_+)$ be a compact Riemannian band with  $n\le 7$ associated with a smooth function
	\begin{equation*}
		\rho : (M,\partial_{\pm}) \to ([-D,D], \pm D)
	\end{equation*}
	satisfying 
	\begin{itemize}
		\item[(1)] $\Lip (\rho ) < 1;$
		\item[(2)] $(M,g)$ has its $T^*$-stabilized scalar curvature bound below by $\sigma$, where $\sigma:M\to \mathbb R$ is a function such that ${\rm{Sc}}(g) \geq c$ in $\rho^{-1}([-1,1])$ and ${\rm{Sc}}(g) \geq 0$ everywhere in $M$
		\item[(3)] Any closed hypersurface $\Sigma$ in $M$ which separates $\partial_-$ from $\partial_+$ admits no metric with positive $T^*$-stabilized scalar curvature.
	\end{itemize}
	Then we have $D \le D_0$ where $D_0=D_0(c) >1$ is a positive constant depending only on $c$.
\end{lemma}

\begin{proof}
For  a constant $c>0$,  one can construct a smooth function $h_c:(-D_0,D_0)\to \mathbb R$ as in the proof of Lemma \ref{WidthEstimate} satisfying that 
\begin{itemize}
\item $h'_c\le0$ and
$$
\lim_{t\to\pm D_0}h_c(t)=\mp\infty;
$$
\item $2h_c'+h_c^2+c\chi_{[-1,1]}>0$, where $\chi_{[-1,1]}$ is the characteristic function of the interval $[-1,1]$.
\end{itemize}
In the following, we are going to deduce a contradiction when $D\geq D_0$. 
By Definition \ref{wrapped}, there are finitely many smooth positive functions $u_1,\ldots,u_k$ on $M$ such that the warped metric
$$
g_{warp}=g_{M}+\sum_{i=1}^k u_i^2\mathrm d\theta_i^2\mbox{ on }M\times \mathbb T^k
$$
satisfies ${\rm{Sc}} (g_{warp})\geq \sigma\circ \pi_M$, where $\pi_M:M\times \mathbb T^k\to M$ is denoted to be the canonical projection map. Suppose that 
$$M_{D_0}:=\rho^{-1}([-D_0,D_0])\times \mathbb T^k$$
and consider the functional
$$
\mathcal A^{h_c}(\Sigma \times \mathbb T^k):=\mathcal H^{k+1}(\Sigma\times \mathbb T^k)-\int_{\Omega\times \mathbb T^k}h_c\circ  \rho\,\mathrm d\mathcal H^{k+2},
$$
where $\Sigma$ is any closed hypersurface separating two boundaries $\rho^{-1}(\pm D_0)$ and bounds a region $\Omega$ with $\rho^{-1}(-D_0)$. From geometric measure theory, we can find a smooth minimizer $\Sigma_{min}\times \mathbb T^k$ of the functional $\mathcal A^{h_c}$. Using the stability and applying the warping product trick by Fischer-Colbrie and Schoen \cite{FC-S1980}, we can find a warped metric
$$
\hat g_{warp}=g_{\Sigma_{min}}+\sum_{i=1}^{k+1}u_i^2\mathrm d\theta_i^2\mbox{ on }\Sigma_{min}\times \mathbb T^{k+1}
$$
with positive scalar curvature, which contradicts our assumptions.
\end{proof}

\begin{corollary}\label{CloseCase}
Assume that $(\Sigma^2,g_\Sigma)$ is a complete and orientable connected surface having its $T^*$-stabilized scalar curvature bounded below by a positive function $\sigma$. Then $\Sigma^2$ is conformally equivalent to a plane or a sphere.
\end{corollary}

\begin{proof}
The proof is divided into two steps:

\vspace{2mm}

\noindent\textbf {Step 1:} $\Sigma$ is topologically a plane or a sphere.

\vspace{2mm}

From the uniformization theorem, it is sufficient to prove the simply-connectedness of $\Sigma$. 

We first note that $H_1(\Sigma) =0$. If not, according to Lemma \ref{Covering}, we can find a covering of $\Sigma$ with a map $\rho: \tilde \Sigma \to \mathbb R$ and $\Lip (\rho) < 1$. For simplicity of notation, we also denote this covering space by $\Sigma$. Both conditions in Lemma \ref{TStableWidthEstimate} hold for the map $\rho : \rho^{-1}([-D,+D]) \to [-D,+D]$ where $D$ is arbitrarily large. This contradicts with Lemma \ref{TStableWidthEstimate}.

We have two cases: (1) $\Sigma$ is closed; (2) $\Sigma$ is open. 
\begin{itemize}

\item[(1)] If $\Sigma$ is closed, then it is $\mathbb{S}^2$. 

\item[(2)] If $\Sigma$ is open, then $\Sigma$ is aspherical. By the Proposition 2.45 in \cite{2002Hat}, $\pi_1(\Sigma)$ is torsion-free. We have  that $\pi_1 (\Sigma)$ is trivial.  If not, we can find a covering space $\tilde \Sigma$ such that $H_1(\tilde \Sigma) \cong \mathbb Z$, and use the same argument as above to lead to a contradiction with Lemma \ref{TStableWidthEstimate}.

\end{itemize}

\noindent \textbf{ Step 2:} $\Sigma$ is conformally equivalent to a plane or a sphere.

\vspace{2mm}

It is sufficient to study the non-compact case. 
Suppose by contradiction that $\Sigma$ is not conformally equivalent to the plane but to the disk $\mathbb D$. Then we can write the metric $g_\Sigma$ as 
$$
g_\Sigma=\mu(z)|\mathrm dz|^2\mbox{ with }z\in\mathbb D.
$$
Choose  $v=\mu^{-1/2}$ and the Gaussian curvature $K$ of $(\Sigma,g_\Sigma)$ can be expressed as follows: 
\begin{equation}\label{Eq: curvature identity}
K=\frac{\Delta_{\Sigma} v}{ v}-\frac{|\nabla_{\Sigma}v|^2}{v^2}.
\end{equation}

In the following, we use the uniform positivity of $T^*$-stability scalar curvature to study the positivity of the operator $-\Delta_\Sigma+\beta_k K$ where $\beta_k=\frac{k+1}{2k}$ and then use it to get a contradiction. 

\vspace{2mm}

\noindent\textbf{Claim: } For  any non-zero $\phi\in C_c^\infty(\Sigma)$, one has that 
\begin{equation}\label{Eq: positive operator}
\int_\Sigma|\nabla_\Sigma\phi|^2+\beta_k K\phi^2\,\mathrm d\sigma_\Sigma>0.
\end{equation}
\emph{Proof of the claim: }Let
$
u=\prod_{i=1}^k u_i.
$ Then, we have   that 
$${\rm{Sc}}(g_{warp})=2K-2\sum_{i=1}^k\frac{\Delta_\Sigma u_i}{u_i}-2\sum_{1\leq i<j\leq k}\langle\nabla_\Sigma\log u_i,\nabla_\Sigma \log u_j\rangle>0$$ 
We observe that 
\[
\begin{split}
\Delta_\Sigma \log u&=\frac{\Delta_\Sigma u}{u}-|\nabla_\Sigma \log u|^2\\
&=\sum_{i=1}^k \frac{\Delta_\Sigma u_i}{u_i}+2\sum_{1\leq i<j\leq k}\langle\nabla_\Sigma\log u_i,\nabla_\Sigma \log u_j\rangle-|\nabla_\Sigma \log u|^2\\
&<K+\sum_{1\leq i<j\leq k}\langle\nabla_\Sigma\log u_i,\nabla_\Sigma \log u_j\rangle-|\nabla_\Sigma \log u|^2.
\end{split}
\]
Notice that for $k\geq 2$ we have
\[
\begin{split}
&|\nabla_\Sigma \log u|^2\\=&\sum_{i=1}^k|\nabla_\Sigma \log u_i|^2+2\sum_{1\leq i<j\leq k}\langle\nabla_\Sigma\log u_i,\nabla_\Sigma \log u_j\rangle\\
=&\sum_{1\leq i< j\leq k}\frac{1}{k-1}(|\nabla_\Sigma \log u_i|^2+|\nabla_\Sigma \log u_j|^2)+2\langle\nabla_\Sigma\log u_i,\nabla_\Sigma \log u_j\rangle\\
\geq &\frac{2k}{k-1}\sum_{1\leq i<j\leq k}\langle\nabla_\Sigma\log u_i,\nabla_\Sigma \log u_j\rangle.
\end{split}
\]
we combine them to get that 
$$
\Delta_{\Sigma}\log u<K-\frac{k+1}{2k}|\nabla_\Sigma \log u|^2=K-\beta_k |\nabla_\Sigma \log u|^2.
$$
Let $\phi$ be any non-zero function in $C_c^\infty(\Sigma)$. Multiplying $\phi^2$ from both sides and integrating by parts,  we see
$$
-\int_{\Sigma}2\phi\langle \nabla_\Sigma\phi,\nabla_\Sigma\log u\rangle\mathrm d\sigma_\Sigma<\int_\Sigma K\phi^2\mathrm d\sigma_\Sigma-\beta_k\int_\Sigma |\nabla_\Sigma\log u|^2\phi^2\,\mathrm d\sigma_\Sigma.
$$
The left-hand side is bounded below by 
$$
-\beta_k^{-1}\int_\Sigma|\nabla_\Sigma\phi|^2\,\mathrm d\sigma_\Sigma-\beta_k\int_\Sigma |\nabla_\Sigma\log u|^2\phi^2\,\mathrm d\sigma_\Sigma
$$
Then, we get \eqref{Eq: positive operator} and complete the proof of the claim.

\vspace{2mm}

We now use  \eqref{Eq: positive operator} to get a contradiction.  We have that 
$$
\int_\Sigma |\nabla_\Sigma(\phi v)|^2+\beta_k K\phi^2 v^2\mathrm d\sigma_\Sigma >0\mbox{ for any non-zero }\phi\in C_c^\infty(\Sigma).
$$
In addition, from \eqref{Eq: curvature identity}, we have that 
\[
\begin{split}
&\int_\Sigma |\nabla_\Sigma(\phi v)|^2=\int_{\Sigma}|\nabla_{\Sigma}\phi|^2v^2\mathrm d\sigma_\Sigma+2\int_{\Sigma}\phi v\langle\nabla_{\Sigma}\phi,\nabla_{\Sigma}v\rangle\mathrm d\sigma_\Sigma +\int_{\Sigma}\phi^2|\nabla_{\Sigma}v|^2\,\mathrm d\sigma_{\Sigma}\\
&\int_{\Sigma }\beta_k K\phi^2 v^2\mathrm = -2\beta_k\int_{\Sigma}\phi v\langle\nabla_{\Sigma}\phi,\nabla_{\Sigma}v\rangle\mathrm-2\beta_k\int_{\Sigma}\phi^2|\nabla_{\Sigma}v|^2\mathrm d\sigma_\Sigma.\\
\end{split}
\]
Using the inequality
\[
\begin{split}
2(1-\beta_k)\int_{\Sigma}&\phi v\langle\nabla_{\Sigma}\phi,\nabla_{\Sigma}v\rangle\mathrm d\sigma_\Sigma\\
&\leq \frac{2(1-\beta_k)^2}{2\beta_k-1}\int_\Sigma|\nabla_\Sigma \phi|^2v^2\mathrm d\sigma_\Sigma+\frac{2\beta_k-1}{2}\int_\Sigma\phi^2|\nabla_\Sigma v|^2\mathrm d\sigma_\Sigma
\end{split}
\]
we arrive at
\[
\int_\Sigma \phi^2|\nabla_\Sigma v|^2\mathrm d\sigma_\Sigma\leq  C(\beta_k)\int_{\Sigma}|\nabla_{\Sigma}\phi|^2v^2\mathrm d\sigma_\Sigma,
\]
where
$$
C(\beta_k)=\frac{2}{2\beta_k-1}\left(\frac{2(1-\beta_k)^2}{2\beta_k-1}+1\right).
$$
Fix a point $p$ in $\Sigma$. By taking $\phi$ to be a test function such that $0\leq \phi\leq 1$ in $\Sigma$, $\phi\equiv 1$ in $B_R(p)$, $\phi\equiv 0$ outside $B_{2R}(p)$ and $|\nabla_\Sigma \phi|\leq 10R^{-1}$, we obtain
$$
\int_{B_R(p)}|\nabla_\Sigma v|^2\mathrm d\sigma_\Sigma\leq 100\pi C(\beta_k)R^{-2}.
$$
Sending $R\to+\infty$ we conclude that $v$ turns out to be a constant function, which contradicts the fact that $g_\Sigma$ is a complete metric.
\end{proof}

\section{Proof of Theorem \ref{main}}

\begin{lemma}\label{FillRadius}
	Let $(\Sigma,\partial\Sigma,g_\Sigma)$ be a compact orientable surface whose $T^*$-stabilized scalar curvature is bounded below by a positive function $\sigma$. Let $\Omega \subset \Sigma $ be a compact connected domain, and let $s$ be a number such that
	\begin{enumerate}
		\item[(1)] $\Omega(s)$ does not meet $\partial \Sigma$.
		\item[(2)] $\mathrm{Image}[H_1(\Omega) \to H_1(\Omega(s))] \not\equiv 0$
	\end{enumerate}
	where $\Omega(s) := \{x\in \Sigma: \dist_{\Sigma}(x,\Omega) \leq s  \}$. Then 
	\begin{equation*}
		s \le D_0\left ( \inf_{\Omega(1)} \sigma \right  )
	\end{equation*}
	where $D_0$ is the function coming from Lemma \ref{TStableWidthEstimate}.
\end{lemma}

\begin{proof} Suppose the contrary that $s> D_0$.  Then, we find a curve $\gamma$ in $\Omega$ which is non-trivial  in $H_1(\Omega(s))$. We may assume $\gamma$ separates two boundaries, $\partial^+ \Omega(s)$ and $\partial^-\Omega(s)$, of the domain $\Omega(s)$. (If not, we use the approach in Lemma \ref{Covering} to find a required covering space and then to replace $\Omega_s$. ) 

Let $\rho: \Omega(s) \to \mathbb R$ be a smoothing of the signed distance function to $\gamma$. By rescaling, we can assume $\Lip (\rho) < 1$ and $\rho: (\Omega(s), \partial^+ , \partial^{-})\rightarrow ([-D_0-\epsilon_1, D_0+\epsilon_2], -D_0-\epsilon_1, D_0+\epsilon_2)$, where $\epsilon_i$ is some positive constant. We combine  Lemma \ref{TStableWidthEstimate} with Remark \ref{no-circle-T-scal} to have that 
 \[D_0+\frac{\epsilon_1+\epsilon_2}{2}\leq D_0\]which is a contradiction. \end{proof}

\begin{proposition}\label{Prop: quantitive topology}
For any function $L:[0,+\infty)\to (0,+\infty)$ and a constant $s_0>0$, there is  a positive constant $T_0=T_0(L,s_0)>s_0$ such that if $(\Sigma,\partial\Sigma,g_\Sigma)$ is a compact oriented connected surface with  boundary such that  \begin{itemize}
\item there is  a smooth function
$\rho:(\Sigma,\partial \Sigma) \to([0,T],T)$ with $\Lip( \rho)<1$; 

\item $T>T_0$ and $T$ is a regular value of $\rho$;

\item The $T^*$-stabilized scalar curvature of $(\Sigma, g_{\Sigma})$ is bounded below by a positive function $\sigma$ with $$
\inf_{\rho^{-1}([0,s])}\sigma\geq L(s) \mbox{ for all }s\in [0,T],
$$\end{itemize}
then for any regular value $s\in (0, s_0]$ of $\rho$,  there are finitely many pairwise disjoint embedded disks $D^s_1,\ldots,D^s_l$ in $\rho^{-1}([0,T_0])$ such that
$$
\bigsqcup_{i=1}^l\partial D^s_i\subset  \rho^{-1}(s)\mbox{ and } \rho^{-1}([0,s])\subset \bigsqcup_{i=1}^l \bar D^s_i.
$$
\end{proposition}

\begin{figure}[tbp]
    \includegraphics[width=1.0\linewidth]{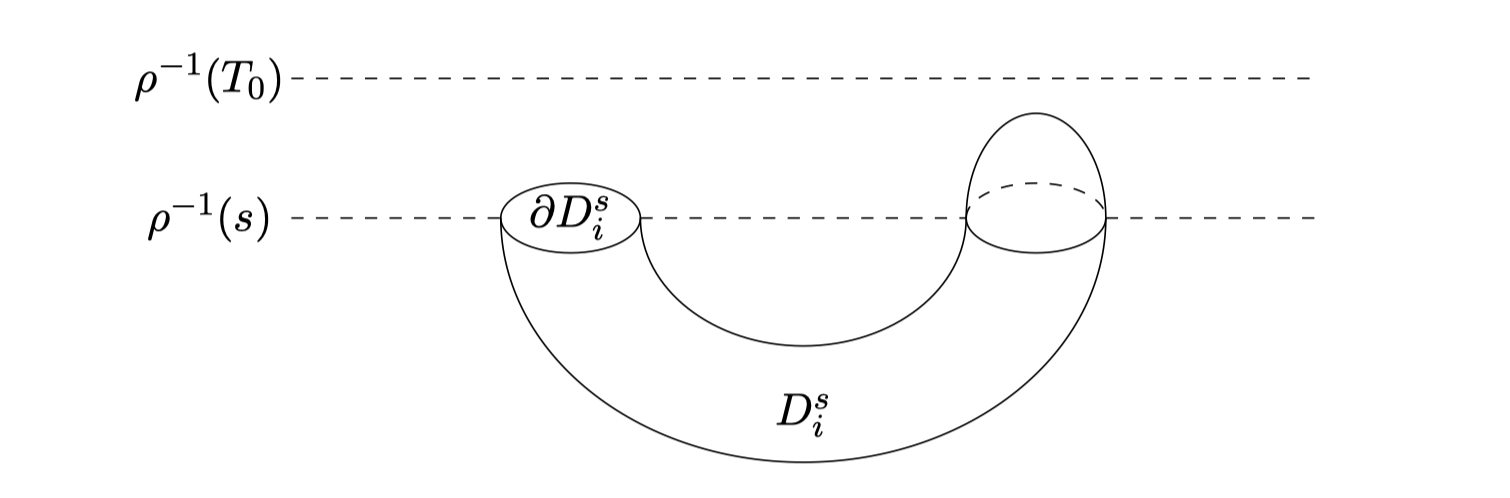}
    \caption{}
\end{figure}

\begin{proof} We first use Lemma \ref{TStableWidthEstimate} to find an expression of $T_0$ and then  show that it is the desired one.

\vspace{2mm}
\noindent \textbf{Step 1:} Find $T_0$ and set up the proof. 
\vspace{2mm}

We use the function $D_0(c)$  defined in Lemma \ref{TStableWidthEstimate} to introduce the following constants: 

\begin{itemize}
\item $C_1=D_0(L(s_0+2))$;
\item $s_1=s_0+1+C_1$;
\item $C_2= D_0(L(s_1+3))$;
\item $T_0=s_1+2+D_2$.
\end{itemize} 

In the following, we assume that  $(\Sigma,\partial\Sigma,g_\Sigma)$ $\sigma$ and $\rho$ are defined as in the assumption. For our convenience, we fix a regular value $s_1^*$ of with $s_1\leq s_1^*\leq s_1+1$.

\vspace{2mm}
\noindent \textbf{ Step 2}: Every component of $\rho^{-1}([0,s_1^*])$ is homeomorphic to a $2$-sphere with finitely many disks removed.
\vspace{2mm}

If not, there is a compact component $\Omega\subset \rho^{-1}([0,s_1^*])$  with $g(\Omega)>0$. Using long exact sequence for relative homology groups (see Page 115 in \cite{2002Hat}), we have that 
for any $C_2>D_0(s^*_1)$\[0<g(\Omega)\leq g(\Omega(C_2))\text{ and }\mathrm{Image} [ H_1(\Omega ) \to H_1(\Omega (C_2))] \neq 0, \]which contradicts lemma \ref{FillRadius}. 

\vspace{2mm}

\noindent\textbf{Step 3: }{For any $s\in (0, s_0]$, each component of $\rho^{-1}(s)$ bounds a disk in $\rho^{-1}([0,s_1^*])$.} 

\vspace{2mm}

It is sufficient to show that $\pi_1(\rho^{-1}(s))\rightarrow \pi_1(\rho^{-1}([0, s^{*}_1]))$ is trivial. 

\vspace{2mm}

Suppose the contrary that there is a closed curve $\gamma$ which is non-contractible in $\rho^{-1}([0, s_1^*])$. For our convenience, we may assume that $\rho^{-1}([0, s^*_1])$ is connected. From Step 2, it is a $2$-sphere with finitely many discs removed.  

We conclude that $\gamma$ separates two non-empty parts $ \partial^+$ and $\partial^-$ of the boundary $\partial [\rho^{-1}([0, s_1^*])]$. If not, $\gamma$ is bounded by a disc, which is a contradiction with the choice of $\gamma$. 

\vspace{2mm}

We now apply Lemma \ref{TStableWidthEstimate} to $(\rho^{-1}([0, s_1^*]), \partial^+, \partial^-)$ and  get a contradiction as follows. 

\vspace{2mm}

 Notice that we have
\begin{itemize}
\item $\gamma\subset \rho^{-1}(s)$,
\item $\partial [\rho^{-1}([0, s_1^*])]=\partial^+ \amalg \partial^- \subset \rho^{-1}(s_1^*)$
\item $\Lip(\rho)<1$
\item $\dist_g(\gamma,\partial^+)>C_1>D_0(L(s_0+2))$ and $\dist_g(\gamma,\partial^-)>C_1>D_0(L(s_0+2))$
\end{itemize}
Through a modification of the signed distance function of $\gamma$, we can construct a smooth function
$$
\tilde \rho:((\rho^{-1}([0, s_1^*]), \partial^+, \partial^-))\to ([-C_1,C_1], C_1, -C_1)
$$
with  $\Lip(\tilde\rho)<1$ and 
$
\tilde \rho^{-1}([-1,1])\subset \gamma(2)\subset \rho^{-1}([0,s_0+2]).
$
Moreover, the scalar curvature is bounded below by  $\sigma \geq L(s_0+2)$ in $\tilde \rho^{-1}([-1,1])$ and is non-negative in $\rho^{-1}([0, s_1^*])$. 

From Remark \ref{wrapped}, any circle has no metric with positive $T^*$-stabilized scalar curvature. We apply Lemma \ref{TStableWidthEstimate} to $(\rho^{-1}([0, s_1^*]), \partial^+, \partial^-) $and  have that  that $C_1<D_0(L_{s_0+2})$, which is a contradiction with the choice of $C_1$.

\vspace{2mm}

\noindent \textbf{Step 4:} Show that for $s\in (0, s_0]$, there are finitely many pairwise disjoint embedded disks $D^s_1,\ldots,D^s_l$ in $\Sigma$ such that
$$
\bigsqcup_{i=1}^l\partial D^s_i\subset \rho^{-1}(s)\mbox{ and } \rho^{-1}([0,s])\subset \bigsqcup_{i=1}^l \bar D^s_i.
$$

\vspace{3mm}

For any $s\in (0, s_0]$, assume that $\rho^{-1}(s)$ is a disjoint union of closed curves,$\{\gamma_i^s\}_{i=1}^m$. From Step 3, each $\gamma^s_i$ bounds a disc $D^s_{i}\subset \rho^{-1}([0, s_1^*])$. 

\vspace{2mm}

\noindent \textbf{Claim: }For any $i_1, i_2$, one of the following relations holds: (1) $D^s_{i_1}\cap D^s_{i_2}=\emptyset$; (2) $D^{s}_{i_1}\subset D^s_{i_2}$; (3) $D^s_{i_2}\subset D^s_{i_1}$.

\vspace{2mm}

\noindent\emph{Proof of the claim:} Suppose that $D^s_{i_1}\setminus D^s_{i_2}\neq \emptyset$ and $D^s_{i_2}\setminus D^s_{i_1}\neq \emptyset$. Then, we have that $\gamma^s_{i_1}\cap D^s_{i_2}\neq \emptyset$ and $\gamma^s_{i_1}\cap D^s_{i_2}\neq \emptyset$. 

Since $\gamma^s_{i_1}$ and $\gamma^s_{i_2}$ are disjoint, we get that $\gamma^s_{i_1}\subset D^s_{i_2}$ and  $\gamma^s_{i_2}\subset D^s_{i_1}$. Namely, the set $D^s_{i_1}\cup D^s_{i_2}\subset \Sigma$ is a $2$-sphere. It is an open and closed set, which is a contradiction with the fact that $\Sigma$ is connected and $\partial \Sigma \neq \emptyset$. 

\vspace{2mm}

The claim suggests  that $(\{D^s_i\}^m_{i=1}, \subset)$ is a partially ordered set. Consider the collection $\{D^s_{i_k}\}^l_{k=1}$ of maximal elements. Then, $\{D^s_{i_k}\}^l_{k=1}$ are disjoint and $\amalg_{k=1}^l D^s_{i_k}=\cup_{i=1}^m D^s_i$. 

\vspace{2mm}

 It remains to show that $\rho^{-1}([0, s])\subset \amalg_{k=1}^l D^s_{i_k}$. 
 
 Suppose by contradiction that there is a component $\mathcal C\subset \rho^{-1}([0, s])$ with ${\rm{Int}} (\mathcal C)\cap (\amalg_{k=1}^l D^s_{i_k})=\emptyset$. The boundary $\partial \mathcal{C}=\cup_{i\in I}\gamma^{s}_{i}$ is a subset of $\cup_{i=1}^m \gamma^s_i$.

 \vspace{1mm}

 We have that for $i, i'\in I$, $D^s_i\cap D^s_{i'}=\emptyset$. If not, use the claim that $D^s_i\subset D^s_{i'}$ or $D^s_{i'}\subset D^s_{i}$. Without the loss of generalization, we assume that $D^s_{i'}\subset D^s_{i}$ (in fact $\bar D_{i'}^s \subset D_i^s$ as $\partial D_{i'}^s \cap \partial D_i^s = \emptyset $).  Since $\rm{Int} (\mathcal C)\cap (\amalg_{k=1}^l D^s_{i_k})=\emptyset$, we get that $\dist (D_{i'}^s, \mathcal{C})>0$, which is a contradiction with $\gamma^s_{i'}\subset \partial \mathcal{C}$. 
 
 \vspace{1mm}
 
 Consider a closed 2-submanifold as follows: 
 \[\mathcal{C}\bigcup_{i\in I} (\cup_{\gamma^s_i} D^s_i)\subset \Sigma\]
 It is a proper non-empty subset that is both open and closed. Since it is a contradiction with the connectedness of $\Sigma$. We complete the proof. 
\end{proof}

Now we are ready to prove Theorem \ref{main}.

\begin{proof}[Proof of Theorem \ref{main}]

Denote 
$$
\tau:=[M]\frown(\beta_1\smile\ldots\smile \beta_{n-2})\in H_2(M, \partial_{s} M),
$$
where $\beta_1 \in H^1(M,\partial_b M)$ and $\beta_2, \ldots, \beta_{n-2} \in H^1(M)$.

\vspace{2mm}

Suppose by contradiction that $(M, g)$ is a complete manifold with positive scalar curvature. It is sufficient to show that $\tau$ is spherical. 

\vspace{2mm}

The proof is divided into two cases.

\textbf{ Case A :} If $\partial_b M = \emptyset$, $\beta_1 \in H^1(M)$ and  $ \tau \in H_2(M, \partial_{\infty} M)$.

Fixing a compact subset $\Omega$ of $M$ and modifying certain distance function, we can construct the following terms
\begin{itemize}
\item $\rho:M\to [0,+\infty)$ is a function with $\rho|_{\Omega}\equiv 0$ and $\Lip(\rho)< 1$.
\item  $\Omega_s:=\rho^{-1}([0,s])$ { for any } $s\geq 0$. 
\item $L(s):=\min_{\Omega_s}R(g)$.
\end{itemize}

Notice that for any $s>0$
\[\tau|_{\Omega_s} \text{ is aspherical and non-zero in } H_2 (M, M\setminus \Omega_s). \]
Where we say a relative class is aspherical if it is not in the image of the Hurewicz map.

\vspace{2mm}

In the following, we  fix a positive constant $s_0$ and  $T>T_0=T_0(L,s_0)$, where $T_0(L, s_0)$ is the constant coming from Proposition \ref{Prop: quantitive topology}. Then, $\tau|_{\Omega_T} ${ is aspherical and non-zero in } $H_2 (M, M\setminus \Omega_T)$, which will lead to a contradiction with the following claim. 

\vspace{2mm}

\noindent {\textbf{Claim: }} If $Sc_M>0$, then $\tau|_{\Omega_T}$ is spherical.

\vspace{2mm}

Since $\tau|_{\Omega_T}$ is non-zero, we can construct a stable weighted slicing with free boundary 
$$
(\Sigma_2,\partial\Sigma_2,w_2)\to (\Sigma_3,\partial\Sigma_3,w_3)\to\ldots \to(\Sigma_n,\partial\Sigma_n,w_n)=(\Omega_T,\partial \Omega_T,1),
$$
where 
\begin{itemize}
\item each $(\Sigma_j,\partial \Sigma_j,w_j)$ is a compact manifold $(\Sigma_j,\partial\Sigma_j)$ with boundary  associated with a  positive smooth function $w_j:\Sigma_j\to\mathbb R$ such that
$$
[\Sigma_j,\partial\Sigma_j]=[\Omega_T,\partial \Omega_T]\frown(e^*\beta_1\smile\ldots\smile e^*\beta_{n-j})
$$
where $e : \Omega_T \hookrightarrow M$ is the inclusion map.
\item for each $3\leq j\leq n$, $\Sigma_{j-1}$ is an embedded two-sided hypersurface with free boundary in $\Sigma_{j}$ with integer multiplicity, which is a stable critical point of the $w_j$-weighted area given by
$$\mathcal A_{j}(\Sigma)=\int_{\Sigma} w_j \,\mathrm d\mathcal H^{j-1}_g;$$
\item for each $3\leq j\leq n$, the quotient function
$$q_{j-1}:=\frac{w_{j-1}}{w_j|_{\Sigma_{j-1}}}$$ is a first eigenfunction of the stability operator on $\Sigma_{j-1}$ with the Neumann boundary condition associated with the $w_j$-weighted area.
\end{itemize}

\vspace{2mm}
\noindent \textbf{Step 1:} The existence of stable weighted slicings with free boundary.
\vspace{2mm}

The main method is essentially contained in Theorem 4.6 in \cite{schoen2017positive}. We will use the inductive construction.

Suppose that  the first $k$ slicings have been constructed. That's to say, 
$$
 (\Sigma_{n-k+1},\partial\Sigma_{n-k+1},w_{n-k+1})\to\ldots \to(\Sigma_n,\partial\Sigma_n,w_n)=(\Omega_T,\partial \Omega_T,1).
$$
where   the weight function $w_{n-k+1}$ is defined as above.

\vspace{2mm}
There is  a smooth hypersurface $(\Sigma^*_{n-k}, \partial \Sigma^*_{n-k}) \subseteq (\Sigma_{n-k+1}, \partial \Sigma_{n-k+1})$ which belongs to the  Poincaré dual of $i^* ( \beta_{n-k})$, i.e.
\begin{equation*}
	[\Sigma^*_{n-k}, \partial \Sigma^*_{n-k}] = [\Sigma_{n-k+1}, \partial \Sigma_{n-k+1}] \frown i^* (\beta_{n-k}).
\end{equation*}where $i: \Sigma_{n-k+1}\rightarrow M$ is the embedding.  Notice that $\Sigma^*$ is a smooth hyperspace, instead of minimal hypersurface. 

\vspace{2mm}

Consider  the collection $\mathcal S_{n-k}$   of $(n-k)$-currents $T$ with the following properties: 
\begin{itemize}
\item $T$ is an $(n-k)$-integral rectifiable current in $\Sigma_{n-k+1}$;
\item For some integral current $U$ and $R$, one has that  \[T-\Sigma^*_{n-k}=\partial S-R,\] where $\mathrm{spt}(S) \subset  \Sigma_{n-k+1} \mbox{ and } \mathrm{spt }(R) \subset \partial \Sigma_{n-k+1} $  
\end{itemize}
consider the variation problem
\begin{equation*}
	I_{n-k} = \inf \{ \mathbb M_{w_{n-k+1}, n-k} (T) : T \in \mathcal S_{n-k} \}
\end{equation*}
where $\mathbb M_{w_{n-k+1}, n-k}$ is the $w_{n-k+1}$-weighted mass functional on $(n-k)$-integer rectifiable currents. 

The standard theory of integral currents \cite{simon2014introduction} allows us to find a minimizer $\Sigma_{n-k}\in \mathcal{S}_{n-k}$ with $I_{n-k} =  \mathbb M_{w_{n-k+1}, n-k} (\Sigma_{n-k}) $. The regularity results for minimizing currents with free boundary (see  \cite{gruter1987optimal, gruter1987regularity, li2017minmax}) show that $\Sigma_{n-k}$ is a smooth, two-side and embedded hypersurface with  $[\Sigma_{n-k}, \partial \Sigma_{n-k}]=[\Omega_T,\partial \Omega_T]\frown(e^*\beta_1\smile\ldots\smile e^*\beta_{n-k})$. 

\vspace{3mm}

Furthermore,  $\Sigma_{n-k}$ satisfies a weighted stability inequality for the weighted function $w_{n-k+1}$.  Therefore we use the stability inequality to find a positive first eigenfunction $q_{n-k}$ of the weighted stability operator. Defining the weight $w_{n-k}$ by the formula $w_{n-k}= q_{n-k}\cdot w_{n-k+1}|_{\Sigma_{n-k}}$ completes the induction step.

\vspace{2mm}
\noindent \textbf{Step 2:} Analyze the topological structure of $\Sigma_2$.
\vspace{2mm}

Using the stability of all $\Sigma_j$ and the warping trick consecutively (The detail can be found in \cite[Section 2.4]{Gr2023}), we have the inequality
$$
{\rm{Sc}}\left(g_{\Sigma_2}+\sum_{j=2}^{n-1}q_j^2\mathrm d\theta_j^2\right)\geq {\rm{Sc}}(g)|_{\Sigma_2}>0.
$$

Now we analyze each component $\mathcal C$ of $\Sigma_2$. In the following, we assume that $\mathcal C$ has a boundary and it touches $\Omega$.(If $\mathcal C$ is closed (and so it is a $2$-sphere by Corollary \ref{CloseCase}) or $\mathcal C$ has boundary without touching $\Omega$, then $[\mathcal C]|_{\Omega}$ is automatically spherical.) 

Consider the surface $(\Sigma_2,\partial\Sigma_2)$ associated with the smooth function $\rho|_{\Sigma_2}$. Let $s$ be a common regular value of $\rho$ and $\rho|_{\Sigma_2}$ with $s<s_0$. It follows from Proposition \ref{Prop: quantitive topology} that there are finitely many pairwise disjoint embedded disks $D_1,\ldots,D_l$ in $\Sigma_2$ such that
$$
\bigsqcup_{i=1}^l \partial D_i\subset \mathcal C\cap\partial \Omega_s\mbox{ and }\mathcal C\cap \Omega_s\subset \bigsqcup_{i=1}^l D_i.
$$
In particular, $[\mathcal C]|_{\Omega_s}$ is spherical and so is $[\mathcal C]|_{\Omega}$.

\textbf{ Case B}: If $\partial_b M \neq \emptyset$,  the only difference from the previous proof is that we find the first slicing $\Sigma_{n-1}$ via the $\mu$-bubble technique. Since 
	\begin{equation*}
		H^{1}(M,\partial_b M) = \varinjlim_{Z\subset M} H^1(M, M \setminus Z )
	\end{equation*}
where $Z$ takes all closed subsets in $M$ satisfing $Z \Delta \partial_s M  \Subset M$. As in lemma \ref{pd1} , choosing a large enough closed subset $Z$, we can find a class $\tilde \beta_1 \in H^1(M, M \setminus Z)$ represents the class $\beta_1 \in H^1(M,\partial_b M)$ and a map $f : (M, M\setminus Z) \to (\mathbb S^1, *)$. Moreover the induced map 
\begin{equation*}
	f^* : H^1(\mathbb S^1,*) \to H^1(M, M-Z)  \to H^1(M, \partial_b M)
\end{equation*}
satisfies $f^*(d \theta) = \beta_1$, where $d\theta$ is the generator of $H^1(\mathbb S^1,*)$

Consider the covering pull-back
\begin{center}
	\begin{tikzcd}
		\tilde M\arrow{r}{\tilde{f}}\arrow{d}{p_1}	&\mathbb{R} \arrow{d}{p_2}\\
		M\arrow{r}{f} &\mathbb S^1
	\end{tikzcd}
\end{center}
where $\tilde M$ is the connected component containing $\tilde \Sigma = \tilde f^{-1}(0)$. Moreover, $(p_1)_*([\tilde \Sigma]) = D_M(\beta_1) \in H_{n-1}(M,\partial_s M)$. Lifting to the covering $\tilde M$, we can assume $\partial_b M = \partial_- \cup \partial_+ $ consist of two ends and there is a hypersurface $\Sigma$ with $[\Sigma] = [M] \frown \beta_1$ separating $\partial_-$ and $\partial_+$. Therefore for any closed subset $\Omega$ with $\Omega \Delta \partial_b M$ is compact, we can construct a stable $\mu$-bubble $(\Sigma_{n-1}, \partial \Sigma_{n-1})$ such that
\begin{equation*}
	[\Sigma_{n-1}, \partial \Sigma_{n-1}] = [\Omega_T, \partial \Omega_T] \frown e^* \beta_1.
\end{equation*}
As we get the first slicing, the rest is exactly the same as before.
\end{proof}

\section{Examples of open SYS manifolds}\label{Sec: examples}

In this section, we will give some examples of open SYS manifolds. Since we are dealing with non-compact manifolds, let us
define what it means for a quasi-proper map between non-compact manifolds
to have non-zero degree.

\begin{definition}{\rm (Cf. \cite[Definition 1.6]{chen2023})} \label{quaiproper}
		Let $M^{n}$ and $N^{n}$ be orientable $n$-manifolds (possibly non-compact). A quasi-proper map $f:M\to N$ is said to have degree n if
	\begin{itemize}\setlength{\itemsep}{1mm}
	
	\item $S_\infty$ consists of discrete points, where
 $$S_\infty=\bigcap_{K\subset M \text{ compact}} \overline{f(M-K)};$$
	\item the composed map
	$$H_n(M, \partial_{\infty})\xrightarrow{i_*} H_n(f^{-1}(N-S_\infty), \partial_{\infty})\xrightarrow{f_*} H_n(N-S_\infty,\partial_{\infty})$$
	maps the fundamental class $[M]$ to $n[N -S_{\infty}]$. In other words,
	\begin{equation*}
		f_*([M]) = n [N-S_{\infty}].
	\end{equation*}
	\end{itemize}
\end{definition}

\begin{proposition}
If $M$ is a domination of an open SYS manifold, then $M$ is also SYS, where we say $M$ is a domination of $\underline{M}$ if there exists a quasi-proper map $f: M \to \underline M$ with $\deg f =\pm 1$. 
\end{proposition}
\begin{proof}
{ Definition \ref{quaiproper} allows us to  find a set  $S_{\infty} = \{\underline p_i\in \underline M ~|~i\in I\}$ of discrete points satisfying }for any proper map $\phi:\mathbb R^+\to \underline M$ the composed map $f\circ \phi$ is either proper or converges to some point $\underline p_i\in S_\infty$ as $t\to+\infty$. By definition of SYS, there exists $\underline V \in \mathcal B(\underline M)$ such that we can find $\underline \beta_1\in H^1(\underline M, \partial_\infty \underline V^c)$ and $\underline \beta_2,\ldots, \underline\beta_{n-2}$ in $H^1(\underline M)$ such that the class
$$
\underline \tau:=[\underline M]\smile(\underline\beta_1\frown\underline \beta_2\frown\cdots\frown\underline \beta_{n-2})\in H_2(\underline M,\partial_\infty \underline V)
$$
is aspherical. Perturb the boundary of $\underline V$, we can assume $\underline p_i \notin  \partial \underline V$. Now we can take $V:=f^{-1}(\underline V)$ which appears to be an element in $\mathcal B(M)$. Since $\underline \beta_1$ is a cohomology class of dimension $1$, its restriction near points $\underline p_i$ is zero. Through pull-back we have
$$
\beta_1:=f^*(\underline \beta_1)\in H^1(M,\partial_\infty V^c)
$$
and
$$
\beta_i:=f^*(\underline \beta_i)\in H^1(M)\mbox{ when }i\neq 1.
$$
We claim that the class
$$
\tau:=[M]\smile(\beta_1\frown\beta_2\frown \cdots\frown \beta_{n-2})\in H_2(M,\partial_\infty V)
$$
is aspherical. Otherwise, there is a subset $\Omega$ of $M$ such that $\overline{\Omega \Delta V}$ is compact and that the class $\tau|_{\Omega}$ is in the image of the Hurewicz map $\pi_2(M,M\setminus \Omega)\to H_{2}(M,M\setminus \Omega)$. Since $f$ has degree $\pm 1$, by push-forward we obtain
$$
\underline \tau|_{f(\Omega)}=\pm f_*(\tau|_{\Omega})
$$
and in particular $\underline \tau|_{f(\Omega)}$ lies in the image of the composed map 
$$\pi_2(\underline M,f( M\setminus \Omega))\to \pi_2(\underline M,\underline M\setminus f(\Omega))\to H_2(\underline M,\underline M\setminus f(\Omega)).$$
This means that $\underline\tau $ is aspherical, which leads to a contradiction.
\end{proof}

\begin{proposition}[i.e. Example \ref{Expl2}]
If $M$ is a closed SYS manifold, then $M \times \mathbb R$ is an open SYS manifold.
\end{proposition}
\begin{proof}
Let $[M] \frown\beta_1 \frown\ldots \frown\beta_{n-2}$ be the aspherical homology class in the definition of closed SYS manifold. Take $V=\emptyset$, where $V$ is the set required in Definition \ref{SYSopen}. Denote $\gamma\in H^1_c(M\times \mathbb R) = H^1(M \times \mathbb R, \partial_{\infty}V^c)$ to be the Poincar\'e dual of $[M]$ in $M\times \mathbb R$. Then the class
$$
[M\times \mathbb R] \frown \gamma \frown\beta_1 \frown\ldots \frown\beta_{n-2}\mbox{ in }H_2(M\times \mathbb R)
$$
is aspherical. We already show that $M\times \mathbb R$ is an open SYS manifold.
\end{proof}

\begin{proposition}[i.e. Example \ref{Expl3} Case (ii)]
		Let $M^n$ be a closed SYS manifold, $\Gamma$ be an embedding submanifold of $M$ with the first Betti number $b_1 \le n-3$, then $M^n \setminus \Gamma$ is an open SYS manifold.
	\label{thm betti}
\end{proposition}

\begin{proof}
	Let $[M] \frown\beta_1 \frown\cdots \frown\beta_{n-2}$ be the aspherical homology class in the definition of closed SYS manifold. Since $ H^1(\Gamma)$ is torsion-free (see universal coefficient theorem in \cite{2002Hat}) and  $\rank H^1(\Gamma) = b_1 \le n-3$, we conclude that $\beta_1|_{\Gamma},\ldots,\beta_{n-2}|_{\Gamma}$ are $\mathbb{Z}$-linearly dependent in $H^1(\Gamma)$, i.e. there are  integers $a_1,\ldots,a_{n-2}$ such that
	\begin{equation*}
		a_1\cdot\beta_1|_{\Gamma}	 + a_2\cdot \beta_2|_{\Gamma} + \ldots +a_{n-2}\cdot \beta_{n-2}|_{\Gamma} = 0 \mbox{ in } H^1(\Gamma). 
	\end{equation*}
	Since $H^1(\Gamma)$ is free, we can assume $(a_1,\ldots,a_{n-2})=1$. That is to say,  there are integers $\mu_i$ such that $\Sigma \mu_i a_i =1$.
	
	We define a new cohomology class 
	\begin{equation*}
		\gamma = a_1\cdot\beta_1	 + a_2\cdot \beta_2 + \cdots +a_{n-2}\cdot \beta_{n-2} \mbox{ in }H^1(M).
	\end{equation*}
	Notice that $\gamma|_{\Gamma} = 0$.
	Since we have the decomposition
	\begin{align*}
		\beta_1\smile \cdots\smile\beta_{n-2} &= \left ( \Sigma\mu_i a_i \right) \cdot \beta_1\smile \cdots \smile \beta_{n-2} \\
		&= \mu_1 \cdot \gamma  \smile \cdots \smile \beta_{n-2} + \cdots +\mu_{n-2}\cdot \beta_1\smile \cdots \smile \gamma, 
	\end{align*}
	we have that at least one component should be aspherical. 
	
Without loss of generality, we may just assume $[M] \frown (\gamma \smile \cdots \smile \beta_{n-2})$ be an aspherical class in $H_2(M)$. { For $\Gamma$ is a closed submanifold of M, there is a long exact sequence} 
	\begin{equation*}
		\cdots \to H_c^i(M\setminus \Gamma) \xrightarrow{j_*} H_c^i(M) \xrightarrow{i^*} H_c^i(\Gamma) \to H^{i+1}_c(M\setminus \Gamma) \to \cdots ,
	\end{equation*}
	{This exact sequence follows from the exact sequence $0\rightarrow C^*_c (M\setminus N_{\epsilon}(\Gamma))\rightarrow C^*_c(M)\rightarrow C^*_c (N_\epsilon(\Gamma))\rightarrow 0$, where $N_\epsilon(\Gamma)$ is the tubular neighborhood of $\Gamma$ with radius $\epsilon$.}
	Thanks to this, we have 
	\begin{equation*}
		0 \to H^1_c(M\setminus \Gamma) \xrightarrow{j_*} H^1(M) \xrightarrow{i^*} H^1(\Gamma).
	\end{equation*}
	Therefore, $\gamma$ can be regarded as a class in $H^1_c(M \setminus \Gamma) = H^1(M \setminus \Gamma, \partial_{\infty}V^c)$ for $V=\emptyset$, where $V$ is the set required in Definition \ref{SYSopen}, and so $M \setminus \Gamma$ is an open SYS manifold.
\end{proof}

Consider  a tubular neighborhood $N(\Gamma)$ of a closed $k$-submanifold $\Gamma$ in a closed manifold $M^n$, 
where $k<n$. The exact sequence for the relative homology  (see \cite{2002Hat}) can be expressed as follows: 
 \begin{center}
 \begin{tikzcd}
H_2(N(\Gamma))\arrow[r, "i_*"] &H_2(M)\arrow[r, "j_*"]& H_2(M, N(\Gamma))\arrow[r, "\partial"]& H_1(N(\Gamma)). 
\end{tikzcd}
\end{center}  
For Borel-Moore homology, there is a long exact localization sequence
\begin{center}
	\begin{tikzcd}
		H^{BM}_i(\Gamma)\arrow[r, "i_*"] &H^{BM}_i(M)\arrow[r, "j_*"]& H^{BM}_i(M/ \Gamma)\arrow[r, "\partial"]& H^{BM}_{i-1}(\Gamma). 
	\end{tikzcd}
\end{center}
Since $\Gamma$ and $M$ are both closed, we have $H^{BM}_*(\Gamma)=H_*(\Gamma)$ and $H^{BM}_*(\Gamma)=H_*(\Gamma)$. Denote $\partial_\infty=\partial_\infty(M\setminus \Gamma)$, then we see $H^{BM}_*(M/\Gamma)=H_*(M\setminus\Gamma,\partial_\infty)$. Moreover, we have the following commutative diagram
\begin{center}
	\begin{tikzcd}
		H_i(\Gamma)\arrow[r, "i_*"] \arrow[d, "\simeq"] &H_i(M)\arrow[r, "j_*"] \arrow[d, equal]& H_i(M\setminus \Gamma, \partial_{\infty})\arrow[r, "\partial"]  \arrow[d, "\simeq"]& H_{i-1}(\Gamma)  \arrow[d, "\simeq"] \\
		H_i(N(\Gamma))\arrow[r, "i_*"] &H_i(M)\arrow[r, "j_*"]& H_i(M, N(\Gamma))\arrow[r, "\partial"]& H_{i-1}(N(\Gamma)) ,
	\end{tikzcd}
\end{center}
where the first and fourth isomorphisms hold since $\Gamma$ is the deformation retract of $N(\Gamma)$ and the third isomorphism comes from the five lemma.

\begin{lemma} Let $M$ be a closed $n$-manifold and $\Gamma$ be a closed $k$-submanifold of $M$ satisfying 
\begin{itemize}
\item $H_1(\Gamma)\cong \pi_1(\Gamma)$, 
\item and $H_2(\Gamma)\cong \pi_2(\Gamma)$.
\end{itemize}
If $\tau$ is an element in $H_2(M)$ such that $j_*(\tau)$ is spherical in $H_2(M, N(\Gamma))$, 
then $\tau$ is also spherical in $H_2(M)$.
\end{lemma}

\begin{proof} We have the following commutative diagram
 \begin{center}\begin{tikzcd}
H_2(N(\Gamma))\arrow[r, "i_*"]&H_2(M)\arrow[r, "j_*"]&H_2(M, N(\Gamma))\arrow[r, "\partial"] &H_1(N(\Gamma)) \\ 

\pi_2(N(\Gamma))\arrow[u, "f_1"]\arrow[r, "i'_*"]&\pi_2(M)\arrow[u, "f_2"]\arrow[r, "j'_*"]&\pi_2(M, N(\Gamma))\arrow[r, "\partial'"]\arrow[u, "f_3"]&\pi_1(N(\Gamma))\arrow[u, "f_4"]
\end{tikzcd}
\end{center}  
where each $f_i$ is the corresponding Hurewicz map. From our assumption both $f_1$ and $f_4$ are isomorphisms. 
The proof depends on the diagram-chasing argument. 

Since $j_*(\tau)$ is spherical in $H_2(M, N(\Gamma))$,  by definition we can find an element $u\in \pi_2(M, N(\Gamma))$ satisfying
$j_*(\tau)=f_3(u)$.
In particular, we obtain $$f_4\circ \partial' (u)=\partial\circ f_3(u)=\partial \circ j_* (\tau)=0.$$ Because $f_4$ is an isomorphism, we have $\partial' (u)=0$.
Due to the exact sequence in the second line, there is an element $v_1\in \pi_2(M)$ such that 
$u=j'_*(v_1)$.
Clearly we have $$j_*\circ f_2(v_1)=f_3\circ  j'_*(v_1)=f_3(u)=j_*(\tau),$$ and so $j_*(\tau-f_2(v_1))=0$.
Now we use the exact sequence in the first line to find $v_2\in H_2(N(\Gamma))$ satisfying
$\tau-f_2(v_1)=i_*(v_2)$.
Recall that the map $f_1: \pi_2(N(\Gamma))\rightarrow H_2 (N(\Gamma))$ is an isomorphism, then there is an element $v'_2\in \pi_2(N(\Gamma))$ satisfying 
$v_2=f_1(v'_2)$.
And we arrive at 
$$f_2\circ i'_*(v'_2)=i_*\circ f_1 (v'_2)=i_*(v_2)=\tau-f_2(v_1).$$ Therefore, we obtain
$\tau=f_2(v_1+i'_*(v'_2))$, which means that $\tau$ is spherical in $H_2(M)$. 
\end{proof}

\begin{corollary}\label{asph} Let $M$ and $\Gamma$ be assumed as above and $\tau\in H_2(M)$. If $\tau$ is aspherical in $H_2(M)$, then $j_*(\tau)$ is aspherical in $H_2(M, N(\Gamma))$.  
\end{corollary}

In the following, let $M$ be a closed SYS manifold. By definition there are cohomology classes $\beta_1 \ldots, \beta_{n-2}\in H^1(M)$ such that 
\[\tau:=[M]\frown (\beta_1\smile \ldots \smile \beta_{n-2})\in H_2(M) \text{ is an aspherical class.}\]
Consider a closed $k$-submanifold $\Gamma$ satisfying
\begin{itemize}
\item $\pi_1(\Gamma)\cong H_1(\Gamma)$,
\item  and $\pi_2(\Gamma)\cong H_2(\Gamma)$. 
\end{itemize}
Notice that we have the following commutative diagram involving cup and cap operations:

  \begin{center}\begin{tikzcd} 
 H_n(M\setminus \Gamma, \partial_{\infty}) \arrow[d,"\simeq"]  &\times &(H^1(M \setminus \Gamma))^{n-2} \arrow[r,"\frown"]& H_2( M\setminus \Gamma, \partial_{\infty}) \arrow[d,"\simeq"]& \\
 
H_n(M, N(\Gamma))&\times &(H^1(M))^{n-2}\arrow[u,"i^{*}"]\arrow[r,"\frown"]& H_2(M, N(\Gamma))& \\ 

H_n(M)\arrow[u, "j_{*}"]&\times&(H^1(M))^{n-2}\arrow[u,equal]\arrow[r, "\frown"]&H_2(M)\arrow[u, "j_{*}"].&
\end{tikzcd}
\end{center}  
As a consequence of Corollary \ref{asph}, we conclude that 
\[j_{*}(\tau) = [M, N(\Gamma)] \frown (\beta_1\smile \ldots \smile \beta_{n-2})\]
is aspherical.

\begin{proposition} Let $M$ be a closed SYS $n$-manifold and $\Gamma$ be a closed embedded $k$-submanifold satisfying
\begin{itemize}
\item $\pi_1(\Gamma)\cong H_1(\Gamma)$,
\item  and $\pi_2(\Gamma)\cong H_2(\Gamma)$,
\end{itemize} where $k<n$. Then $M\setminus \Gamma$ is an open SYS manifold. 
\end{proposition}

\begin{proof} Let $N(\Gamma)$ be the same as above. We shall show that 
\begin{equation*}
	u=[M \setminus \Gamma] \frown (i^{*}(\beta_1)\smile \ldots \smile i^{*}(\beta_{n-2})) \in H_2(M\setminus \Gamma, \partial_{\infty})
\end{equation*}
is aspherical. Take $V=M \setminus \Gamma$, where $V$ is the set required in Definition \ref{SYSopen}. 

{
Choose $\Omega = M \setminus N(\Gamma)$. The commutative diagram above gives that the map $H_2(M\setminus \Gamma, \partial_{\infty}) \to H_2(M,N(\Gamma))\cong  H_2(M, M\setminus \Omega)$ maps $u$ to $j_*(\tau)$ (i.e. $u|_{\Omega} = j_*(\tau))$. 
}

Since  $j_*(\tau)$ is aspherical by the Corollary \ref{asph}, $u|_\Omega$ is aspherical as well, which yields that $M\setminus \Gamma$ is an open SYS manifold.
\end{proof}

\begin{remark}
	If $\Gamma$ is simply-connected or $\Gamma$ is a curve, then $\Gamma$ satisfies the topological conditions $\pi_1(\Gamma)\cong H_1(\Gamma)$ and $\pi_2(\Gamma)\cong H_2(\Gamma)$.
\end{remark}
\begin{corollary}[i.e. Example \ref{Expl3} Case (i)]
Let $M^n$ be a closed SYS manifold, $\Gamma$ be an embedding submanifold of $M$ with $\dim \Gamma \leq 1$, then $M^n \setminus \Gamma$ is an open SYS manifold.
\end{corollary}

\section{Uniformly positive scalar curvature}

In this section, we study the UPSC obstructions on an open manifold. Similar to Theorem \ref{main}, we prove that a weak SYS manifold does not admit a complete UPSC metric.

\begin{theorem}
	For $3 \le n \le 7$ if $M^n$ is weakly SYS, then it admits no complete metrics with uniformly positive scalar curvature. 
\end{theorem}

\begin{proof}
	The proof is almost the same as the proof of theorem \ref{main}.
	Denote 
$$
\tau:=[M]\frown(\beta_1\smile\ldots\smile \beta_{n-2})\in H_2(M, \partial_{s} M),
$$
where $\beta_1 \in H^1(M,\partial_b M)$ and $\beta_2, \ldots, \beta_{n-2} \in H^1(M)$.

\vspace{2mm}

	Assume $M$ admits a complete metric with $Sc \ge \sigma > 0$. For any closed subset $\Omega$ of $M$ satisfying that $\Omega \Delta V^c $ is precompact, we can choose a compact region $K \subset \Omega$ such that $\beta_1|_{\Omega \setminus K} = 0$.

	Then consider the stable weighted slicing  for  $T > \frac{2\pi}{\sqrt{\sigma}}$ as above:
	\begin{equation*}
		(\Sigma_2,\partial\Sigma_2,w_2)\to (\Sigma_3,\partial\Sigma_3,w_3)\to\ldots \to(\Sigma_n,\partial\Sigma_n,w_n)=(K_T,\partial K_T,1).
	\end{equation*}
	And the restricted class
$$
\tau|_{K_T}=[K_T,\partial K_T]\frown(e^*\beta_1\smile\ldots\smile e^*\beta_{n-2})= [\Sigma_2, \partial\Sigma_2],
$$
where $e:\Omega_T\to M$ is the inclusion map.

	{By the $T^*$-stable 2d Bonnet-Myers diameter inequality in \cite[Section 2.8]{Gr2023} and the same argument in the proof of Theorem \ref{main},} each component $\mathcal C$ of $\Sigma_2$ intersecting $K$ is a sphere with $\mathrm{diam}(\mathcal C) \le \frac{2\pi}{\sqrt{\sigma}}$. As $\beta_1|_{\Omega \setminus K} = 0$, the components lying outside of $K$ are trivial in homology. Consequently, 
	$$\tau|_\Omega = \sum_{\mathcal C\cap K \neq \emptyset} [\mathcal C] $$
	is in the image of $\pi_2(M)$ which implies that $\tau$ is spherical. The proof is completed.
\end{proof}

\begin{corollary}
	Let $M$ be a closed SYS manifold with the aspherical class $\tau= [M] \frown (\beta_1 \smile \ldots \smile  \beta_{n-2}) \in H_2(M)$ and $\Gamma \subset M$ be a embedded submanifold. If $\tau$ is not in the subgroup generated by the image of $H_2(\Gamma)$ and $\pi_2(M)$, then $M$ is weakly SYS.
\end{corollary}

\begin{proof}
	We shall show that 
\begin{equation*}
	u=[M \setminus \Gamma] \frown (i^{*}(\beta_1)\smile \ldots \smile i^{*}(\beta_{n-2})) \in H_2(M\setminus \Gamma, \partial_{\infty})
\end{equation*}
is weakly aspherical by contradiction. Where $i: M \setminus \Gamma \to M$ is the natural injection.

Let $N(\Gamma)$ be a tubular neighborhood of $\Gamma$ and $K = M \setminus N(\Gamma)$. If $u$ is strongly spherical, by Definition \ref{Defwclass}, $u|_K = \tau |_K$ is in the image of the Hurewicz map
\begin{equation*}
	\pi_2(M) \to H_2(M) \to H_2(M, M\setminus K).
\end{equation*}
Hence there exists a spherical class $s \in H_2(M)$ such that $\tau|_K - s|_K = 0$ in $H_2(M,M\setminus K)$.

Combining with the following commutative diagram 
\begin{center}
	\begin{tikzcd}
		H_2(\Gamma)\arrow[r, "i_*"] \arrow[d, "\simeq"] &H_2(M)\arrow[r] \arrow[d, equal]& H_2(M\setminus \Gamma, \partial_{\infty})\arrow[r, "\partial"]  \arrow[d, "\simeq"]& H_1(\Gamma)  \arrow[d, "\simeq"] \\
		H_2(N(\Gamma))\arrow[r, "i_*"] &H_2(M)\arrow[r]& H_2(M, M \setminus K)\arrow[r, "\partial"]& H_1(N(\Gamma)) ,
	\end{tikzcd}
\end{center}
$\tau - s $ is in the image of $H_2(\Gamma)$, so $\tau$ lies in the group generated by $H_2(\Gamma)$ and $\pi_2(M)$. Thus we have arrived at a contradiction.

\end{proof}

\begin{remark}
	Moreover, if we assume that $\Gamma$ is incompressible, i.e., $\pi_1(\Gamma) \to \pi_1(M)$ is injective, then using the same diagram chasing, we can prove that $M \setminus \Gamma$ is an open SYS manifold.
\end{remark}

\appendix
\section{Poincaré Duality Theorem on open manifolds}

Let $M$ be an open manifold with an ends-decomposition $\partial_{\infty} M= \partial_{\infty} A \sqcup \partial_{\infty} B$ with $A,B \in \mathcal B(M)$. Let $M = \bigcup K_i$, $K_i \subset \subset K_{i+1}$, be an exhaustion of $M$.

With loss of generality, we assume that $\partial A = \partial B \subset \mathrm{Int}(K_1)$. Let $A_i = \partial K_i \cap A$ and $B_i = \partial K_i \cap B$, we get a decomposition of $\partial K_i = A_i \sqcup B_i$. 

To deduce the generalization of Poincaré duality to open manifolds, let us recall the version of compact manifolds with boundary.

\begin{theorem}{\rm (Cf. \cite[Theorem 3.43]{2002Hat})}
	Suppose $M$ is a compact orientable $n$-manifold whose boundary $\partial M$ is decomposed as the union of two compact $(n-1)$-dimensional manifolds $A$ and $B$ with a common boundary $\partial A = \partial B = A \cap B$. Then the cap product with a fundamental class $[M] \in H_n(M, \partial M)$ gives isomorphisms for all $k$
	\begin{equation*}
		D_M : H^k(M, A) \to H_{n-k}(M, B), \quad \alpha \mapsto [M] \frown \alpha.
	\end{equation*} 
	
\end{theorem}

We can apply the duality theorem to $(K_i, A_i \sqcup B_i)$ and take a limit in some sense to obtain

\begin{theorem}\label{pd}
	Suppose $M$ is an open orientable manifold with an ends-decomposition $\partial_{\infty} M = \partial_{\infty} A \sqcup \partial_{\infty} B$. Then the cap product with a fundamental class $[M] \in H_n(M, \partial_{\infty} M)$ gives isomorphisms for all $k$
	\begin{equation*}
		D_M : H^k(M, \partial_{\infty} A) \to H_{n-k}(M,\partial_{\infty} B), \quad \alpha \mapsto [M] \frown \alpha.
	\end{equation*} 
	
\end{theorem}
\vspace{2mm}
In order to prove Theorem \ref{pd}, let us recall that the homology groups $H_k(M, \partial_{\infty} B) $ are the homology groups of the chain complex
\begin{equation*}
	C_*(M, \partial_{\infty} B) = \varprojlim_i C_*(M, B \setminus K_i)
\end{equation*} 

Unfortunately, $H_k(M, \partial_{\infty} B)$ is not equal to $\varprojlim H_k(M, B \setminus K_i)$ because  the inverse limit is not an exact functor. However, we have the Milnor exact sequence

\begin{lemma}\label{MilnorSeq}
	The sequence $0 \to \varprojlim^1 H_{k+1}(M,B \setminus K_i) \to H_k(M, \partial_{\infty} B) \to \varprojlim H_k(M,B \setminus K_i) \to 0$ is exact.
\end{lemma}

\begin{proof}
	Since $C_k(M,B \setminus K_{i+1}) \to C_k(M,B \setminus K_i)$ is surjective, it satisfies Mittag-Leffler condition. By the Theorem 3.5.8 in \cite{weibel_1994}, the exact sequence holds. 
\end{proof}

Similarly, for cohomology groups, we have

\begin{lemma}\label{MilnolSeqCo}
	The sequence $0 \to \varprojlim^1 H^{k-1}(A \cup K_i , \partial_{\infty} A) \to H^k(M, \partial_{\infty} A) \to \varprojlim H^k(A \cup K_i, \partial_{\infty} A) \to 0$ is exact.
\end{lemma}

\begin{proof}
	By the Lemma 11.9 in \cite{massey1978homology}, we have the Milnor exact sequence 
	\begin{equation*}
		0 \to \varprojlim \nolimits^{1}  H^{k-1}(A \cup K_i) \to H^k(M) \to \varprojlim H^{k}(A \cup K_i) \to 0. 
	\end{equation*}
	
	Combining with the long exact sequence of relative cohomology groups, the above exact sequence holds for the relative version
	$$
		0 \to \varprojlim \nolimits^{1}  H^{k-1}(A \cup K_i, A \setminus K_i) \to H^k(M, A\setminus K_i) \to \varprojlim H^{k}(A \cup K_i, A \setminus K_i) \to 0. 
	$$
	Since the direct limit is an exact functor, the cohomology group $H^k(M, \partial_{\infty} A) = \varinjlim_i H^k(M, A \setminus K_i)$. Taking direct limit, the lemma holds.

\end{proof}

 Now we can prove the Theorem \ref{pd}
 
 \begin{proof}[Proof of Theorem \ref{pd}]
 	Let $M_i$ be the manifold with boundary $A \cup K_i $ and $H_c^k(M_i) = H^k(M_i, \partial_{\infty} A) $. 
 	Consider the following commutative diagram
%

	\begin{equation*}
		\xymatrix{
		0 \ar[r] &\varprojlim^1 H_c^{n-k-1}(M_i) \ar[r] \ar[d]_{[M_i]\frown -}&H^{n-k}(M, \partial_{\infty} A)  \ar[r]\ar[d]_{[M]\frown -}&\varprojlim H_c^{n-k}(M_i)  \ar[r]\ar[d]_{[M_i]\frown -}& 0 \\
		0 \ar[r] &\varprojlim^1 H_{k+1}(M_i,\partial M_i) \ar[r]&H_k(M, \partial_{\infty}B)   \ar[r]&\varprojlim H_k(M_i,\partial M_i)  \ar[r]& 0 
	}
	\end{equation*}
	
	The first row is exact by Lemma \ref{MilnolSeqCo}, and the second row is exact by Lemma \ref{MilnorSeq} and the excision theorem. By the Theorem 11.3 in \cite{massey1978homology}, the Poincaré duality Theorem holds for manifolds with boundary and cohomology with compact support, i.e. the map
	\begin{equation*}
		D_{M_i}= [M_i] \frown - : H_c^{n-k}(M_i) \to H_k(M_i, \partial M_i)
	\end{equation*}
	is isomorphism. Hence the induced maps in the first and third columns are isomorphism. By the five Lemma, the second column is also isomorphism.
\end{proof}

\nocite{*}

\bibliography{Positive}

\def\bysame{\leavevmode ---------\thinspace}
\makeatletter\if@francais\providecommand{\og}{<<~}\providecommand{\fg}{~>>}
\else\gdef\og{``}\gdef\fg{''}\fi\makeatother
\def\cdrandname{\&}
\providecommand\cdrnumero{no.~}
\providecommand{\cdredsname}{eds.}
\providecommand{\cdredname}{ed.}
\providecommand{\cdrchapname}{chap.}
\providecommand{\cdrmastersthesisname}{Memoir}
\providecommand{\cdrphdthesisname}{PhD Thesis}
\begin{thebibliography}{10}

\bibitem{Cecchini2023}
{\scshape S.~Cecchini, D.~Räde {\normalfont \cdrandname}~R.~Zeidler}, {\og
  Nonnegative scalar curvature on manifolds with at least two ends\fg},
  \emph{Journal of Topology} \textbf{16} (2023), \cdrnumero 3, p.~855–876.

\bibitem{chen2021incompressible}
{\scshape J.~Chen, P.~Liu, Y.~Shi {\normalfont \cdrandname}~J.~Zhu}, {\og
  Incompressible hypersurface, positive scalar curvature and positive mass
  theorem\fg}, 2021, \url{https://arxiv.org/abs/2112.14442}.

\bibitem{Chen2022AGO}
{\scshape S.~Chen}, {\og A generalization of the Geroch conjecture with
  arbitrary ends\fg}, \emph{Mathematische Annalen} (2022).

\bibitem{chen2023}
{\scshape S.~Chen, J.~Chu {\normalfont \cdrandname}~J.~Zhu}, {\og Positive
  scalar curvature metric and aspherical summands\fg},  (2023),
  \url{https://arxiv.org/abs/2312.04698}.

\bibitem{2020Generalized}
{\scshape O.~Chodosh {\normalfont \cdrandname}~C.~Li}, {\og Generalized soap
  bubbles and the topology of manifolds with positive scalar curvature\fg},
  \emph{arXiv:2008.11888v3} (2020).

\bibitem{FC-S1980}
{\scshape D.~Fischer-Colbrie {\normalfont \cdrandname}~R.~Schoen}, {\og The
  structure of complete stable minimal surfaces in 3-manifolds of non-negative
  scalar curvature\fg}, \emph{Communications on Pure and Applied Mathematics}
  \textbf{33} (1980), \cdrnumero 2, p.~199-211,
  \url{https://arxiv.org/abs/https://onlinelibrary.wiley.com/doi/pdf/10.1002/cpa.3160330206}.

\bibitem{Gr2023}
{\scshape M.~Gromov}, {\og Four lectures on scalar curvature\fg}, in
  \emph{Perspectives in scalar curvature. {V}ol. 1}, World Sci. Publ.,
  Hackensack, NJ, [2023] \copyright 2023, p.~1-514.

\bibitem{Gr1983}
{\scshape M.~Gromov {\normalfont \cdrandname}~H.~B. Lawson}, {\og Positive
  scalar curvature and the Dirac operator on complete riemannian manifolds\fg},
  \emph{Publications Math{\'e}matiques de l'Institut des Hautes {\'E}tudes
  Scientifiques} \textbf{58} (1983), p.~83-196.

\bibitem{gruter1987optimal}
{\scshape M.~Gr{\"u}ter}, {\og Optimal regularity for codimension one minimal
  surfaces with a free boundary\fg}, \emph{manuscripta mathematica} \textbf{58}
  (1987), p.~295-343.

\bibitem{gruter1987regularity}
\bysame , {\og Regularity results for minimizing currents with a free
  boundary.\fg},  (1987).

\bibitem{hao2023llarull}
{\scshape T.~Hao, Y.~Shi {\normalfont \cdrandname}~Y.~Sun}, {\og Llarull type
  theorems on complete manifolds with positive scalar curvature\fg}, 2023,
  \url{https://arxiv.org/abs/2310.10173}.

\bibitem{2002Hat}
{\scshape A.~Hatcher}, {\og Algebraic Topology\fg}, \emph{Cambridge University
  Press} (2002).

\bibitem{He2023}
{\scshape S.~He {\normalfont \cdrandname}~J.~Zhu}, {\og A note on rational
  homology vanishing theorem for hypersurfaces in aspherical manifolds\fg},
  2023.

\bibitem{kazaras2019}
{\scshape D.~Kazaras}, {\og Desingularizing positive scalar curvature
  4-manifolds\fg}, \emph{arXiv preprint arXiv:1905.05306} (2019).

\bibitem{LUY2020}
{\scshape M.~Lesourd, R.~Unger {\normalfont \cdrandname}~S.-T. Yau}, {\og
  Positive Scalar Curvature on Noncompact Manifolds and the Liouville
  Theorem\fg}, 2020, \url{https://arxiv.org/abs/2009.12618}.

\bibitem{LM19}
{\scshape C.~Li {\normalfont \cdrandname}~C.~Mantoulidis}, {\og Positive scalar
  curvature with skeleton singularities\fg}, \emph{Math. Ann.} \textbf{374}
  (2019), \cdrnumero 1-2, p.~99-131.

\bibitem{li2017minmax}
{\scshape M.~Li {\normalfont \cdrandname}~X.~Zhou}, {\og Min-max theory for
  free boundary minimal hypersurfaces I - regularity theory\fg}, 2017,
  \url{https://arxiv.org/abs/1611.02612}.

\bibitem{massey1978homology}
{\scshape W.~Massey}, \emph{Homology and Cohomology Theory: An Approach Based
  on Alexander-Spanier Cochains}, Monographs and textbooks in pure and applied
  mathematics, M. Dekker, 1978.

\bibitem{schick1998counterexample}
{\scshape T.~Schick}, {\og A counterexample to the (unstable)
  Gromov--Lawson--Rosenberg conjecture\fg}, \emph{Topology} \textbf{37} (1998),
  \cdrnumero 6, p.~1165-1168.

\bibitem{schoen1979existence}
{\scshape R.~Schoen {\normalfont \cdrandname}~S.-T. Yau}, {\og Existence of
  incompressible minimal surfaces and the topology of three dimensional
  manifolds with non-negative scalar curvature\fg}, \emph{Annals of
  Mathematics} \textbf{110} (1979), \cdrnumero 1, p.~127-142.

\bibitem{SY1979}
\bysame , {\og On the structure of manifolds with positive scalar
  curvature\fg}, \emph{Manuscripta Math.} \textbf{28} (1979), \cdrnumero 1-3.

\bibitem{schoen1982complete}
{\scshape R.~Schoen {\normalfont \cdrandname}~S.~T. Yau}, {\og Complete
  three-dimensional manifolds with positive Ricci curvature and scalar
  curvature\fg}, in \emph{Seminar on differential geometry}, vol. 102,
  Princeton University Press Princeton, 1982, p.~209-228.

\bibitem{schoen2017positive}
{\scshape R.~Schoen {\normalfont \cdrandname}~S.-T. Yau}, {\og Positive Scalar
  Curvature and Minimal Hypersurface Singularities\fg}, 2017,
  \url{https://arxiv.org/abs/1704.05490}.

\bibitem{simon2014introduction}
{\scshape L.~Simon}, {\og Introduction to geometric measure theory\fg},
  \emph{Tsinghua Lectures} \textbf{2} (2014), \cdrnumero 2, p.~3-1.

\bibitem{wang2024scalarcurvaturerigidityspheres}
{\scshape J.~Wang {\normalfont \cdrandname}~Z.~Xie}, {\og Scalar curvature
  rigidity of spheres with subsets removed and $L^\infty$ metrics\fg}, 2024,
  \url{https://arxiv.org/abs/2407.21312}.

\bibitem{wz2022}
{\scshape X.~Wang {\normalfont \cdrandname}~W.~Zhang}, {\og On the generalized
  {G}eroch conjecture for complete spin manifolds\fg}, \emph{Chinese Ann. Math.
  Ser. B} \textbf{43} (2022), \cdrnumero 6, p.~1143-1146.

\bibitem{weibel_1994}
{\scshape C.~A. Weibel}, \emph{An Introduction to Homological Algebra},
  Cambridge Studies in Advanced Mathematics, Cambridge University Press, 1994.

\bibitem{Zhu21}
{\scshape J.~Zhu}, {\og Width estimate and doubly warped product\fg},
  \emph{Trans. Amer. Math. Soc.} \textbf{374} (2021), \cdrnumero 2,
  p.~1497-1511.

\end{thebibliography}
\end{document}